\documentclass{amsart}
\usepackage{fullpage,amsfonts,amsmath,amssymb,graphicx,color,mathpazo}
\theoremstyle{plain}
\newtheorem{thm}{Theorem}[section]
\newtheorem{lemma}[thm]{Lemma}
\newtheorem{prop}[thm]{Proposition}
\newtheorem{rhp}{Riemann-Hilbert Problem}
\newtheorem{dbarproblem}{$\dbar$ Problem}
\newcommand{\dbar}{{\overline{\partial}}}
\DeclareMathOperator{\im}{Im}
\DeclareMathOperator{\re}{Re}
\newcommand{\ii}{\mathrm{i}}
\newcommand{\ee}{\mathrm{e}}
\newcommand{\dd}{\mathrm{d}}

\begin{document}
\title{Dispersive Asymptotics for Linear and Integrable Equations by the $\bar{\partial}$ Steepest Descent Method}
\author{Momar Dieng}
\address{Department of Mathematics, University of Arizona
}%
\author{Kenneth D. T.-R. McLaughlin}%
\address{Deptartment of Mathematics, University of Arizona\footnote{Current affiliation:  Department of Mathematics, Colorado State University.  Email:  \texttt{kenmcl@rams.colostate.edu}}}%
\author{Peter D. Miller}%
\address{Department of Mathematics, University of Michigan}
\email{millerpd@umich.edu}

\thanks{The first two authors were supported in part by NSF
grants DMS-0451495, DMS-0800979, and the second author was supported by NSF Grant DMS-1733967.  The third author was supported in part by NSF grant DMS-1812625.}

\begin{abstract}
We present a new and relatively elementary method for studying the solution of the initial-value problem for dispersive linear and integrable equations in the large-$t$ limit, based on a generalization of steepest descent techniques for Riemann-Hilbert problems to the setting of $\dbar$-problems.  Expanding upon prior work \cite{DiengM08} of the first two authors, we develop the method in detail for the linear and defocusing nonlinear Schr\"odinger equations, and show how in the case of the latter it gives sharper asymptotics than previously known under essentially minimal regularity assumptions on initial data.
\end{abstract}

\maketitle

\section{Introduction}

The long time behavior of solutions $q(x,t)$ of the Cauchy initial-value problem for the defocusing nonlinear Schr\"{o}dinger (NLS) equation
\begin{equation}
\label{eq:NLSEQ}
\ii \frac{\partial q}{\partial t} + \frac{\partial^2 q}{\partial x^2} - 2 |q|^2 q = 0,
\end{equation}
with initial data decaying as for large $x$:
\begin{equation}
q(x,0)=q_0(x)\to 0,\quad |x|\to\infty,
\label{eq:IC}
\end{equation}
has been studied extensively, under various assumptions on the smoothness and decay properties of the initial data $q_{0}$ \cite{SegurA76,Zakh1, Its81,Deif3, Deif4, Deif7, Deif5}.  The asymptotic behavior takes the following form: as $t \to +\infty$, one has
\begin{equation}
\label{eq:AsForm}
q(x,t) = t^{-1/2}\alpha(z_0)\ee^{\ii x^2/(4t) - \ii\nu(z_0)\ln(8t)} + \mathcal{E}(x,t),
\end{equation}
where $\mathcal{E}(x,t)$ is an error term and for $z\in\mathbb{R}$, $\nu(z)$ and $\alpha(z)$ are defined by
\begin{equation}
\nu(z) := -\frac{1}{2\pi}\ln(1-|r(z)|^2),\quad |\alpha(z)|^2=\frac{1}{2}\nu(z),
\label{eq:nu-define}
\end{equation}
and
\begin{equation}
\arg(\alpha(z))=\frac{1}{\pi}\int_{-\infty}^{z}\ln(z-s)\,\dd \ln(1-|r(s)|^2)+\frac{\pi}{4} + \arg(\Gamma(\ii\nu(z)))-\arg(r(z)).
\label{eq:argalpha-define}
\end{equation}
Here $z_{0} = - x / (4t)$, $\Gamma$ is the gamma function, and $r(z)$ is the so-called reflection coefficient associated to the initial data $q_{0}$.  The connection between the initial data $q_{0}(x)$ and the reflection coefficient $r(z)$ is achieved through the spectral theory of the associated self-adjoint Zakharov-Shabat differential operator
\begin{equation}
\mathcal{L}:=
\nonumber
\ii  \sigma_3\frac{\dd}{\dd x} + \mathbf{Q}(x),\quad\sigma_3:=\begin{pmatrix}
1 & 0 \\
0 & -1 
\end{pmatrix},\quad\mathbf{Q}(x):=\begin{pmatrix}
0 & - \ii q_0(x) \\
\ii \overline{q_0(x)} & 0 
\end{pmatrix},
\label{eq:ZS}
\end{equation}
acting in $L^2(\mathbb{R};\mathbb{C}^2)$ as described, for example, in \cite{Deif4}.  

The modulus $|\alpha(z_0)|$ of the complex amplitude $\alpha(z_0)$ as written in \eqref{eq:nu-define} was first obtained by Segur and Ablowitz \cite{SegurA76} from trace formul\ae\ under the assumption that $q(x,t)$ has the form \eqref{eq:AsForm} where $\mathcal{E}(x,t)$ is small for large $t$.  Zakharov and Manakov \cite{Zakh1} took the form \eqref{eq:AsForm} as an ansatz to motivate a kind of WKB analysis of the reflection coefficient $r(z)$ and as a consequence were able to also calculate the phase of $\alpha(z_0)$, obtaining for the first time the phase as written in \eqref{eq:argalpha-define}.  Its \cite{Its81} was the first to observe the key role played in the large-time behavior of $q(x,t)$ by an ``isomondromy'' problem for parabolic cylinder functions; this problem has been an essential ingredient in all subsequent studies of the large-$t$ limit and as we shall see it is a non-commutative analogue of the Gaussian integral that produces the familiar factors of $\sqrt{2\pi}$ in the stationary phase approximation of integrals.  The first time that the form \eqref{eq:AsForm} itself was rigorously deduced from first principles (rather than assumed) and proven to be accurate for large $t$ (incidentally reproducing the formul\ae\ \eqref{eq:nu-define}--\eqref{eq:argalpha-define} in an ansatz-free fashion) was in the work of Deift and Zhou \cite{Deif3} (see \cite{Deif4} for a pedagogic description) who brought the recently introduced nonlinear steepest descent method \cite{Deif6} to bear on this problem.  Indeed, under the assumption of  high orders of smoothness and decay on the initial data $q_0$, the authors of \cite{Deif3} proved that $\mathcal{E}(x,t)$ satisfies
\begin{equation}
 \sup_{x\in\mathbb{R}}|\mathcal{E}(x,t)| = \mathcal{O} \left( \frac{\ln(t)}{t} \right),\quad t\to +\infty.
 \label{eq:DZ-estimate-1}
\end{equation}

It is reasonable to expect that any estimate of the error term $\mathcal{E}(x,t)$ would depend on the smoothness and decay assumptions made on $q_{0}$, and so it is natural to ask what happens to the estimate \eqref{eq:DZ-estimate-1} if the assumptions on $q_0$ are weakened. Early in this millennium, Deift and Zhou developed some new tools for the analysis of Riemann-Hilbert problems, originally aimed at studying the long time behavior of perturbations of the NLS equation \cite{Deif8}.  Their methods allowed them to establish long time asymptotics for the Cauchy problem \eqref{eq:NLSEQ}--\eqref{eq:IC} with essentially minimal assumptions on the initial data \cite{Deif5}.  Indeed, they assumed  the initial data $q_0$ to lie in the weighted Sobolev space 
\begin{equation}
H^{1,1}(\mathbb{R}):=\left\{f\in L^2(\mathbb R):\ xf,f'\in L^2(\mathbb R)\right\}.
\label{eq:H11-def}
\end{equation}
It is well known that if $q_{0} \in H^{1,1}(\mathbb{R})$, then the associated reflection coefficient\footnote{Since $q_0\in H^{1,1}(\mathbb{R})$ implies that $(1+|x|)q_0(x)$ is square-integrable, it follows by Cauchy-Schwarz that $H^{1,1}(\mathbb{R})\subset L^1(\mathbb{R})$, which in turn implies that the reflection coefficient $r(z)$ is well-defined for each $z\in\mathbb{R}$.} satisfies $r \in H^{1,1}_{1}(\mathbb{R})$, where
\begin{equation}
H^{1,1}_1(\mathbb{R}):=\left\{f\in H^{1,1}(\mathbb{R}):\  \sup_{z\in\mathbb{R}}|f(z)|<1\right\},
\label{eq:H111-def}
\end{equation}
and more generally the spectral transform $\mathcal{R}$ associated with the Zakharov-Shabat operator $\mathcal{L}$ \eqref{eq:ZS} is a map $\mathcal{R}: H^{1,1}(\mathbb{R}) \to H^{1,1}_{1}(\mathbb{R})$, $q_{0} \mapsto r=\mathcal{R}q_0$ that is a bi-Lipschitz bijection \cite{Zhou1}.  
The result of \cite{Deif5} is then that the Cauchy problem \eqref{eq:NLSEQ}--\eqref{eq:IC} for $q_0\in H^{1,1}(\mathbb{R})$ has a unique weak solution for which  \eqref{eq:AsForm} holds with an error term $\mathcal{E}\left(x,t \right)$  that satisfies, for any fixed $\kappa$ in the indicated range,
\begin{equation}
 \sup_{x\in\mathbb{R}}|\mathcal{E}(x,t)|  = \mathcal{O}\left( t^{- \left( \frac{1}{2} + \kappa \right) } \right),\quad t\to +\infty,\quad 0 < \kappa < \frac{1}{4}. 
\label{eq:DZ-estimate-2}
\end{equation}

Subsequently, McLaughlin and Miller \cite{McLa1,Mcla2} developed a method for the asymptotic analysis of Riemann-Hilbert problems in which jumps across contours are ``smeared out'' over a two-dimensional region in the complex plane, resulting in an equivalent $\dbar$ problem that is more easily analyzed.  In this paper we adapt and extend this method to the Riemann-Hilbert problem of inverse-scattering associated to the Cauchy problem \eqref{eq:NLSEQ}--\eqref{eq:IC}.  The main point of our work is this:  by using the $\dbar$ approach, we avoid all delicate estimates involving Cauchy projection operators in $L^{p}$ spaces (which are central to the work in \cite{Deif5}).  Instead it is only necessary to estimate certain double integrals, an exercise involving nothing more than calculus.  Remarkably, this elementary approach also sharpens the result obtained in \cite{Deif5}.  Our result is as follows.

\begin{thm}
The Cauchy problem \eqref{eq:NLSEQ}--\eqref{eq:IC} with initial data $q_0$ in the weighted Sobolev space $H^{1,1}(\mathbb{R})$ defined by \eqref{eq:H11-def} has a unique weak solution having the form \eqref{eq:AsForm}--\eqref{eq:argalpha-define} in which $r(z)$ is the reflection coefficient associated with $q_0$ and where the error term satisfies
\begin{equation}
\sup_{x\in\mathbb{R}}|\mathcal{E}(x,t)|=\mathcal{O}\left(t^{-\tfrac{3}{4}}\right),\quad t\to +\infty.
\label{eq:improved-estimate}
\end{equation}
\label{mainresult}\end{thm}
The main features of this result are as follows.
\begin{itemize}
\item The error estimate is an improvement over the one reported in \cite{Deif5}, \textit{i.e.}, we prove that the endpoint case $\kappa=\tfrac{1}{4}$ holds in \eqref{eq:DZ-estimate-2}.  Our methods also suggest that the improved estimate \eqref{eq:improved-estimate} on the error is sharp.
\item As with the result \eqref{eq:DZ-estimate-2} obtained in \cite{Deif5}, the improved estimate \eqref{eq:improved-estimate} only requires the condition $r \in H^{1,0}_{1}(\mathbb{R})$, \textit{i.e.}, it is not necessary that $z r(z) \in L^{2}(\mathbb{R})$, but only that $r$ lies in the classical Sobolev space $H^1(\mathbb{R})$ and satisfies $|r(z)|\le\rho$ for some $\rho<1$.  Dropping the weighted $L^2$ condition on $r$ corresponds to admitting rougher initial data $q_0$.  For such data, the solution of the Cauchy problem is of a weaker nature, as discussed at the end of \cite{Deif5}.
\item The new $\dbar$ method which is used to derive the estimate \eqref{eq:improved-estimate} affords a considerably less technical proof than previous results.
\item The method used to establish the estimate \eqref{eq:improved-estimate} is readily extended to derive a more detailed asymptotic expansion, beyond the leading term (see the remark at the end of the paper).
\end{itemize}

Given the reflection coefficient $r\in H^{1,1}_1(\mathbb{R})$ associated with initial data $q_0\in H^{1,1}(\mathbb{R})$ via the spectral transform $\mathcal{R}$ for the Zakharov-Shabat operator $\mathcal{L}$,
the solution of the Cauchy problem for the nonlinear Schr\"{o}dinger equation \eqref{eq:NLSEQ} may be described as follows.  
Consider the following Riemann-Hilbert problem:
\begin{rhp}
\label{rhp:01}
Given parameters $(x,t)\in\mathbb{R}^2$, find $\mathbf{M} = \mathbf{M}(z) = \mathbf{M}(z;x,t)$, a $2 \times 2$ matrix, satisfying the following conditions:
\begin{itemize}
\item[]\textit{\textbf{Analyticity:}} $\mathbf{M}$ is an analytic function of $z$ in the domain $\mathbb{C}\setminus\mathbb{R}$.  Moreover, $\mathbf{M}$ has a continuous extension to the real axis from the upper (lower) half-plane denoted $\mathbf{M}_+(z)$ ($\mathbf{M}_-(z)$) for $z\in\mathbb{R}$.
\item[]\textit{\textbf{Jump condition:}} The boundary values satisfy the jump condition
\begin{equation}
\mathbf{M}_+(z)=\mathbf{M}_-(z)\mathbf{V}_\mathbf{M}(z),\quad z\in\mathbb{R},
\end{equation}
where the jump matrix $\mathbf{V}_\mathbf{M}(z)$ is defined by
\begin{equation}
\mathbf{V}_\mathbf{M}(z):=\begin{pmatrix}1-|r(z)|^2 & -\overline{r(z)}\ee^{-2\ii t\theta(z;z_0)}\\
r(z)\ee^{2\ii t\theta(z;z_0)} & 1\end{pmatrix},\quad z\in\mathbb{R},\quad\theta(z;z_0):=2z^2-4z_0z,\quad z_0:=-\frac{x}{4t}.
\label{eq:basic-jump}
\end{equation}
\item[]\textit{\textbf{Normalization:}}  There is a matrix $\mathbf{M}_1(x,t)$ such that
\begin{equation}
\mathbf{M}(z)=\mathbb{I} + z^{-1}\mathbf{M}_1(x,t) + o(z^{-1}),\quad z\to\infty.
\end{equation}
\end{itemize}
\end{rhp}
From the solution of this Riemann-Hilbert problem, one defines a function $q(x,t)$, $(x,t)\in\mathbb{R}^2$, by
\begin{equation}\label{qMrelation}
q(x,t) := 2 \ii M_{1,12}(x,t).
\end{equation}
The fact of the matter is then that $q(x,t)$ is the solution of the Cauchy problem \eqref{eq:NLSEQ}--\eqref{eq:IC}.  

Recent studies of the long-time behavior of the solution of the NLS initial-value problem \eqref{eq:NLSEQ}--\eqref{eq:IC} have involved the detailed analysis of the solution $\mathbf{M}$ to Riemann-Hilbert problem~\ref{rhp:01}.  As regularity assumptions on the initial data $q_{0}$ are relaxed, this analysis becomes more involved, technically.  The purpose of this manuscript is to carry out a complete analysis of the long-time asymptotic behavior of $\mathbf{M}$ under the assumption that $r \in H^{1,1}_{1}(\mathbb{R})$ (or really, $r\in H^{1,0}_1(\mathbb{R})$), as in \cite{Deif4}, but via a $\overline{\partial}$ approach which replaces technical harmonic analysis estimates involving Cauchy projection operators with very straightforward estimates involving some explicit double integrals.

The proof of Theorem~\ref{mainresult} using the methodology of \cite{McLa1,Mcla2} was originally obtained by the first two authors in 2008 \cite{DiengM08}.  Since then the technique has been used successfully to study many other related problems of large-time behavior for various integrable equations.  In \cite{CuccagnaJ16}, the authors used the methods of \cite{DiengM08} to analyze the stability of multi-dark-soliton solutions of \eqref{eq:NLSEQ}.  In \cite{BorgheseJM18}, the method of \cite{DiengM08} was used to confirm the soliton resolution conjecture for the focusing version of the NLS equation under generic conditions on the discrete spectrum.  In \cite{LiuPS18}, the large-time behavior of solutions of the derivative NLS equation was studied using $\dbar$ methods, and in \cite{JenkinsLPS18} the same techniques were used to establish a form of the soliton resolution conjecture for this equation.  Similar $\dbar$ methods more based on the original approach of \cite{McLa1,Mcla2} have also been useful in studying some problems of nonlinear wave theory not necessarily in the realm of large time asymptotics, for instance \cite{MillerQ15}, which deals with boundary-value problems for \eqref{eq:NLSEQ} in the semiclassical limit.  Based on this continued interest in $\dbar$ methods, we decided to write this review paper containing all of the results and arguments of \cite{DiengM08}, some in a new form, as well as some additional expository material which we hope the reader might find helpful.

\section{An unorthodox approach to the corresponding linear problem}
\label{sec:linear}
In order to motivate the $\dbar$ steepest descent method, we first consider the Cauchy problem for the linear equation corresponding to \eqref{eq:NLSEQ}, namely
\begin{equation}
\ii\frac{\partial q}{\partial t} +\frac{\partial^2 q}{\partial x^2}=0,
\end{equation}
with initial condition \eqref{eq:IC} for which $q_0\in H^{1,1}(\mathbb{R})$.  By Fourier transform theory, if 
\begin{equation}
\hat{q}_0(z):=\int_\mathbb{R}q_0(x)\ee^{2\ii z x}\,\dd x,\quad z\in\mathbb{R}
\end{equation}
is the Fourier transform of the initial data, then $\hat{q}_0$ as a function of $z\in\mathbb{R}$ also lies in the weighted Sobolev space $H^{1,1}(\mathbb{R})$, and the solution of the Cauchy problem is given in terms of $\hat{q}_0$ by the integral
\begin{equation}
q(x,t)=\frac{1}{\pi}\int_\mathbb{R}\hat{q}_0(z)\ee^{-2\ii t\theta(z;z_0)}\,\dd z,
\label{eq:linear-solution}
\end{equation}
where $\theta(z;z_0)$ and $z_0$ are as defined in \eqref{eq:basic-jump}.  It is worth noticing that this formula is exactly what arises from Riemann-Hilbert Problem~\ref{rhp:01} via the formula \eqref{qMrelation} if only the jump matrix $\mathbf{V}_\mathbf{M}(z)$ in \eqref{eq:basic-jump} is replaced with the triangular form
\begin{equation}
\mathbf{V}_\mathbf{M}(z):=\begin{pmatrix}1 & -\hat{q}_0(z)\ee^{-2\ii t\theta(z;z_0)}\\0 & 1\end{pmatrix},\quad z\in\mathbb{R}
\end{equation}
in which case the solution of Riemann-Hilbert Problem~\ref{rhp:01} is explicitly given by
\begin{equation}
\mathbf{M}(z;x,t)=\mathbb{I} -\frac{1}{2\pi\ii}\int_\mathbb{R}\frac{\hat{q}_0(\zeta)\ee^{-2\ii t\theta(\zeta;z_0)}}{\zeta-z}\,\dd\zeta\begin{pmatrix}0 & 1\\0 & 0\end{pmatrix}.
\end{equation} 
This shows that the reflection coefficient $r(z)$ is a nonlinear analogue of (the complex conjugate of) the Fourier transform $\hat{q}_0(z)$.  

Assuming that $z_0\in\mathbb{R}$ is fixed, the method of stationary phase applies to deduce an asymptotic expansion of the integral in \eqref{eq:linear-solution}.  The only point of stationary phase is $z=z_0$, and the classical formula of Stokes and Kelvin yields
\begin{equation}
q(x,t)=\frac{1}{\pi}\sqrt{\frac{2\pi}{t|-2\theta''(z_0;z_0)|}}\hat{q}_0(z_0)\ee^{-2\ii t\theta(z_0;z_0)-\ii\pi/4}+\mathcal{E}(x,t) = t^{-1/2}\frac{\hat{q}_0(z_0)\ee^{-\ii\pi/4}}{2\sqrt{\pi}}\ee^{\ii x^2/(4t)}+\mathcal{E}(x,t),
\label{eq:stationary-phase}
\end{equation}
where the error term is of order $t^{-3/2}$ as $t\to +\infty$ under the best assumptions on $\hat{q}_0$, assumptions that guarantee that the error has a complete asymptotic expansion in terms proportional via explicit oscillatory factors to descending half-integer powers of $t$.  To derive this expansion from first principles consists of several steps as follows.
\begin{itemize}
\item One introduces a smooth partition of unity to separate contributions to the integral from points $z$ close to $z_0$ and far from $z_0$.
\item One uses integration by parts to estimate the contributions from points $z$ far from $z_0$.  This requires having sufficiently many derivatives of $\hat{q}_0(z)$, which corresponds to having sufficient decay of $q_0(x)$.  
\item One approximates $\hat{q}_0(z)$ locally near $z_0$ by an \emph{analytic function} with an accuracy related to the size of $t$ and the number of terms of the expansion that are desired.
\item One uses Cauchy's theorem to deform the path of integration for the approximating integrand to a diagonal path over the stationary phase point.  The slope of the diagonal path produces the phase factor of $\ee^{-\ii\pi/4}$, and the path integral of the leading term $\hat{q}_0(z_0)\ee^{-2\ii t\theta(z;z_0)}$ in the local approximation of $\hat{q}_0(z)\ee^{-2\ii t\theta(z;z_0)}$ is a Gaussian integral that produces the factor of $\sqrt{\pi}$.
\end{itemize}
It is possible to implement all steps of this method assuming, say, that $q_0$ (and hence also $\hat{q}_0$) is a Schwartz-class function.  However, as one reduces the regularity of $q_0$ it becomes impossible to obtain an expansion to all orders.  More to the point, even in the presence of Schwartz-class regularity, the proof of the stationary phase expansion by the traditional methods outlined above is complicated, perhaps more so than necessary as we hope to convince the reader.

To explain an alternative approach that bears fruit in the case $q_0\in H^{1,1}(\mathbb{R})$ that is of interest here, let $\Omega$ denote a simply-connected region in the complex plane with counter-clockwise oriented piecewise-smooth boundary $\partial\Omega$.  If $f:\Omega\to\mathbb{C}$ is differentiable (as a function of two real variables $u=\re(z)$ and $v=\im(z)$) and extends continuously to $\partial\Omega$, then it follows from Stokes' theorem that
\begin{equation}
\oint_{\partial\Omega}f(u,v)\,\dd z = \iint_\Omega 2\ii\dbar f(u,v)\,\dd A(u,v)
\label{eq:Stokes}
\end{equation}
where $\dd A(u,v)$ denotes area measure in the plane and where $\dbar$ is the Cauchy-Riemann operator:
\begin{equation}
\dbar:=\frac{1}{2}\left(\frac{\partial}{\partial u}+\ii\frac{\partial}{\partial v}\right),\quad z=u+\ii v,
\label{eq:dbar}
\end{equation}
which annihilates all analytic functions of $z=u+\ii v$.
Now consider the diagram shown in Figure~\ref{fig:Integral}.
\begin{figure}[h]
\begin{center}
\includegraphics{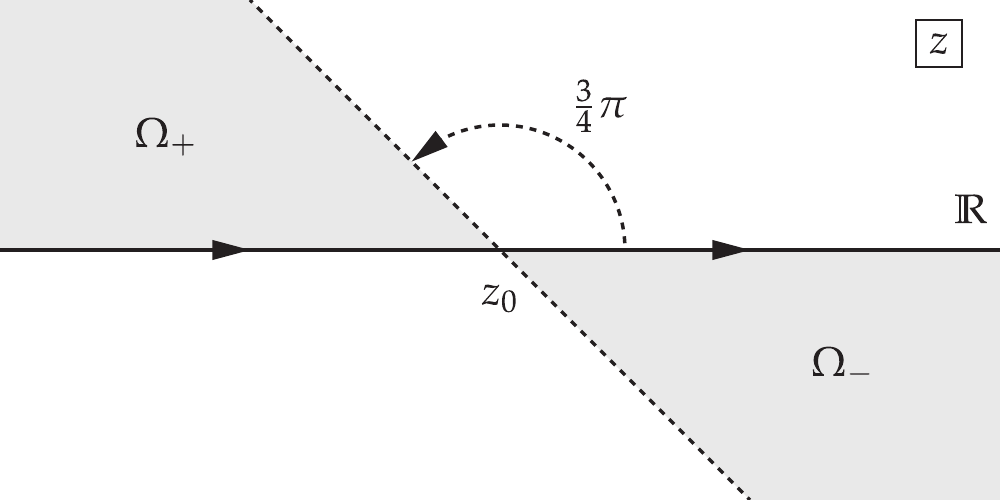}
\end{center}
\caption{The integration contour $\mathbb{R}$ in \eqref{eq:linear-solution} and the unbounded domains $\Omega_+$ and $\Omega_-$ in the $z=u+\ii v$ plane.}
\label{fig:Integral}
\end{figure}
We define a function $E(u,v)$ on $\Omega_+\cup\Omega_-$ as follows:
\begin{equation}
E(u,v):=\cos(2\arg(u+\ii v-z_0))\hat{q}_0(u) + \left(1-\cos(2\arg(u+\ii v-z_0))\right)\hat{q}_0(z_0),\;\; u+\ii v\in\Omega_+\cup\Omega_-.
\label{eq:extension-linear}
\end{equation}
Observe that:
\begin{itemize}
\item On the boundary $v=0$ (\textit{i.e.}, $z\in\mathbb{R}$), we have $\cos(2\arg(u+\ii v-z_0))\equiv 1$, so $E(u,0)=\hat{q}_0(u)$.
\item On the boundary $v=z_0-u$, we have $\cos(2\arg(u+\ii v-z_0))\equiv 0$, so $E(u,z_0-u)=\hat{q}_0(z_0)$ which is independent of $u$.
\end{itemize}
The first point shows that $E(u,v)$ is an \emph{extension} of the function $\hat{q}_0(z)$ from the real $z$-axis into the domain $\Omega_+\cup\Omega_-$.  The second point shows that the extension evaluates to a constant on the diagonal part of the boundary of $\Omega_+\cup\Omega_-$.  In the interior of $\Omega_+\cup\Omega_-$, $E(u,v)$ inherits smoothness properties from $\hat{q}_0(u)$.
In particular, under the assumption $\hat{q}_0\in H^{1,1}(\mathbb{R})$, we may apply Stokes' theorem in the form \eqref{eq:Stokes}
to the functions $\pm E(u,v)\ee^{-2\ii t\theta(u+\ii v;z_0)}$ on the domains $\Omega_\pm$ and add up the results to obtain the formula
\begin{equation}
q(x,t)=\frac{1}{\pi}\int_{\infty\ee^{3\pi\ii/4}}^{\infty\ee^{-\ii\pi/4}}\hat{q}_0(z_0)\ee^{-2\ii t\theta(z;z_0)}\,\dd z +
\frac{1}{\pi}\iint_{\Omega_+-\Omega_-}2\ii\dbar\left(E(u,v)\ee^{-2\ii t\theta(u+\ii v;z_0)}\right)\,\dd A(u,v).
\label{eq:q-Stokes}
\end{equation}
The first term on the right-hand side originates from the diagonal boundary of $\Omega_+\cup\Omega_-$ and because $E$ is constant there it is an exact Gaussian integral evaluating to the explicit leading term on the right-hand side of \eqref{eq:stationary-phase}.  Therefore, the remaining term on the right-hand side of \eqref{eq:q-Stokes} is an exact double-integral representation of the error term $\mathcal{E}(x,t)$ in the formula \eqref{eq:stationary-phase}.  Since $q_0\in H^{1,1}(\mathbb{R})$ implies $\hat{q}_0\in H^{1,1}(\mathbb{R})$ which in turn implies that $\hat{q}_0(z)$ is defined for all $z\in\mathbb{R}$, the leading term in \eqref{eq:stationary-phase} certainly makes sense.  

To estimate the error term we will only use the fact that $\hat{q}_0'\in L^2(\mathbb{R})$, \textit{i.e.}, that $\hat{q}_0$ lies in the (classical, unweighted) Sobolev space $H^1(\mathbb{R})$.  First note that since $\ee^{-2\ii t\theta(z;z_0)}$ is an entire function of $z$, $\dbar\ee^{-2\ii t\theta(z;z_0)}\equiv 0$, so by the product rule it suffices to have suitable estimates of $\dbar E(u,v)$ for $u+\ii v\in\Omega_\pm$.  Indeed, 
\begin{equation}
\begin{split}
\left|\iint_{\Omega_\pm}2\ii\dbar\left(E(u,v)\ee^{-2\ii t\theta(u+\ii v;z_0)}\right)\,\dd A(u,v)\right|&\le
2\iint_{\Omega_\pm}|\dbar E(u,v)|\ee^{2t\im(\theta(u+\ii v;z_0))}\,\dd A(u,v)\\ & = 
2\iint_{\Omega_\pm}|\dbar E(u,v)|\ee^{8t(u-z_0)v}\,\dd A(u,v).
\end{split}
\label{eq:linear-integral-estimate-1}
\end{equation}
A direct computation using \eqref{eq:dbar} gives
\begin{equation}
\begin{split}
\dbar E(u,v)&=\dbar\left[\hat{q}_0(z_0)+\cos(2\arg(u+\ii v-z_0))\left(\hat{q}_0(u)-\hat{q}_0(z_0)\right)\right]\\
&=\cos(2\arg(u+\ii v-z_0))\dbar\hat{q}_0(u) + \left(\hat{q}_0(u)-\hat{q}_0(z_0)\right)\dbar\cos(2\arg(u+\ii v-z_0))\\
&=\frac{1}{2}\cos(2\arg(u+\ii v-z_0))\hat{q}_0'(u) +\left(\hat{q}_0(u)-\hat{q}_0(z_0)\right)\dbar\cos(2\arg(u+\ii v-z_0)).
\end{split}
\end{equation}
In polar coordinates $(\rho,\phi)$ centered at the point $z_0\in\mathbb{R}$ and defined by $u=z_0+\rho\cos(\phi)$ and $v=\rho\sin(\phi)$, the Cauchy-Riemann operator \eqref{eq:dbar} takes the equivalent form
\begin{equation}
\dbar = \frac{\ee^{\ii\phi}}{2}\left(\frac{\partial}{\partial\rho}+\frac{\ii}{\rho}\frac{\partial}{\partial\phi}\right),
\label{eq:dbar-polar}
\end{equation}
so as $\arg(u+\ii v-z_0)=\phi$ we have 
\begin{equation}
\dbar\cos(2\arg(u+\ii v-z_0))=\frac{\ii\ee^{\ii\phi}}{2\rho}\frac{\dd}{\dd\phi}\cos(2\phi)=-\frac{\ii\ee^{\ii\phi}}{\rho}\sin(2\phi).
\end{equation}
Therefore we easily obtain the inequality
\begin{equation}
|\dbar E(u,v)|\le\frac{1}{2}|\hat{q}_0'(u)| + \frac{|\hat{q}_0(u)-\hat{q}_0(z_0)|}{\sqrt{(u-z_0)^2+v^2}},\quad u+\ii v\in\Omega_+\cup\Omega_-.
\label{eq:dbarE-1}
\end{equation}
Note that by the fundamental theorem of calculus and the Cauchy-Schwarz inequality,
\begin{equation}
\begin{split}
|\hat{q}_0(u)-\hat{q}_0(z_0)|\le\int_{z_0}^u|\hat{q}_0'(w)|\,|\dd w|&\le\sqrt{\int_{z_0}^u|\dd w|}\sqrt{\int_{z_0}^u|\hat{q}_0'(w)|^2\,|\dd w|}\\
&\le \|\hat{q}_0'\|_{L^2(\mathbb{R})}\sqrt{|u-z_0|}\le
\|\hat{q}_0'\|_{L^2(\mathbb{R})}\left[(u-z_0)^2+v^2\right]^{1/4},
\end{split}
\label{eq:qhat-difference}
\end{equation}
so \eqref{eq:dbarE-1} implies that also
\begin{equation}
|\dbar E(u,v)|\le\frac{1}{2}|\hat{q}_0'(u)| + \frac{\|\hat{q}_0'\|_{L^2(\mathbb{R})}}{[(u-z_0)^2+v^2]^{1/4}},\quad u+\ii v\in\Omega_+\cup\Omega_-.
\label{eq:dbarE-2}
\end{equation}
Therefore, using \eqref{eq:dbarE-2} in \eqref{eq:linear-integral-estimate-1} gives
\begin{equation}
\left|\iint_{\Omega_\pm}2\ii\dbar\left(E(u,v)\ee^{-2\ii t\theta(u+\ii v,z_0)}\right)\,\dd A(u,v)\right|\le 
I^\pm(x,t) + 2\|\hat{q}_0'\|_{L^2(\mathbb{R})}J^\pm(x,t),
\end{equation}
where
\begin{equation}
I^\pm(x,t):=\iint_{\Omega_\pm} |\hat{q}_0'(u)|\ee^{8t(u-z_0)v}\,\dd A(u,v)
\quad\text{and}\quad J^\pm(x,t):=\iint_{\Omega_\pm}\frac{\ee^{8t(u-z_0)v}}{[(u-z_0)^2+v^2]^{1/4}}\,\dd A(u,v).
\label{eq:linear-integrals}
\end{equation}
The key point is that for $t>0$, the exponential factors are bounded by $1$ and decaying at infinity in $\Omega_\pm$.  So, by iterated integration, Cauchy-Schwarz, and the change of variable $w=t^{1/2}(u-z_0)$,
\begin{equation}
\begin{split}
I^+(x,t)&=\int_{-\infty}^{z_0}\,\dd u\int_0^{z_0-u}\,\dd v\,|\hat{q}_0'(u)|\ee^{8t(u-z_0)v}\\
& = 
\int_{-\infty}^{z_0}\,\dd u |\hat{q}_0'(u)|\frac{1-\ee^{-8t(u-z_0)^2}}{8t(z_0-u)}\\ &\le
\|\hat{q}_0'\|_{L^2(\mathbb{R})}\sqrt{\int_{-\infty}^{z_0}\left[\frac{1-\ee^{-8t(u-z_0)^2}}{8t(z_0-u)}\right]^2\,\dd u}\\
&=K\|\hat{q}_0'\|_{L^2(\mathbb{R})}t^{-3/4},\quad K:=\sqrt{\int_{-\infty}^0\left[\frac{1-\ee^{-8w^2}}{8w}\right]^2\,\dd w}<\infty.
\end{split}
\end{equation}
In exactly the same way, we also get $I^-(x,t)\le K\|\hat{q}_0'\|_{L^2(\mathbb{R})}t^{-3/4}$.  Note that $K$ is an absolute constant.  The integrals $J^\pm(x,t)$ are independent of $q_0$ and by translation of $z_0$ to the origin and reflection through the origin, the integrals are also independent of $x$ and are obviously equal.  To calculate them we introduce rescaled polar coordinates by $u=z_0+t^{-1/2}\rho\cos(\phi)$ and $v=t^{-1/2}\rho\sin(\phi)$ to get
\begin{equation}
J^\pm(x,t)=Lt^{-3/4},\quad L:=\int_0^\infty\,\rho\dd\rho\int_{3\pi/4}^\pi\,\dd\phi \,\rho^{-1/2}\ee^{8\rho^2\sin(\phi)\cos(\phi)}
\end{equation}
It is a calculus exercise to show that the above double integral is convergent and hence defines $L$ as a second absolute constant.  

It follows from these elementary calculations that if only $\hat{q}_0'\in L^2(\mathbb{R})$, then the error term $\mathcal{E}(x,t)$ in \eqref{eq:stationary-phase} obeys the estimate
\begin{equation}
\sup_{x\in\mathbb{R}}|\mathcal{E}(x,t)|\le\frac{2}{\pi}(K+2L)\|\hat{q}_0'\|_{L^2(\mathbb{R})}t^{-3/4}
\end{equation}
which decays as $t\to+\infty$ at exactly the same rate as in the claimed result for the nonlinear problem as formulated in Theorem~\ref{mainresult}.  The same method can be used to obtain higher-order corrections under additional hypotheses of smoothness for the Fourier transform $\hat{q}_0$.  One simply needs to integrate by parts with respect to $u=\re(z)$ in the double integral on the right-hand side of \eqref{eq:q-Stokes}.

In the rest of the paper we will show that almost exactly the same elementary estimates suffice to prove the nonlinear analogue of this result, namely Theorem~\ref{mainresult}.  

\section{Proof of Theorem~\ref{mainresult}}
We will prove Theorem~\ref{mainresult} in several systematic steps.  After some preliminary observations involving the jump matrix $\mathbf{V}_\mathbf{M}(z)$ in Riemann-Hilbert Problem~\ref{rhp:01} in Sections~\ref{sec:factorize} and \ref{sec:diagonal}, we shall see that the subsequent analysis of Riemann-Hilbert Problem~\ref{rhp:01} parallels our study of the associated linear problem detailed in Section~\ref{sec:linear}.  In particular we find natural analogues of the nonanalytic extension method (Section~\ref{sec:extend}), the Gaussian integral giving the leading term in the stationary phase formula (Section~\ref{sec:PC}), and of the simple double integral estimates leading to the proof of its accuracy (Section~\ref{sec:dbar}).  Finally, in Section~\ref{sec:q} we assemble the ingredients to arrive at the formula \eqref{eq:AsForm} with the improved error estimate, completing the proof of Theorem~\ref{mainresult}.

\subsection{Jump matrix factorization}
\label{sec:factorize}
The jump matrix $\mathbf{V}_\mathbf{M}(z)$ of Riemann-Hilbert Problem~\ref{rhp:01} defined in \eqref{eq:basic-jump} can be factored in two different ways that are useful in different intervals of the jump contour $\mathbb{R}$ as indicated:
\begin{equation}
\mathbf{V}_\mathbf{M}(z)=\begin{pmatrix}1&-\overline{r(z)}\ee^{-2\ii t\theta(z;z_0)}\\0 & 1\end{pmatrix}\begin{pmatrix}1 & 0\\r(z)\ee^{2\ii t\theta(z;z_0)} & 1\end{pmatrix},\quad z>z_0,
\label{eq:factorization-right}
\end{equation}
and
\begin{equation}
\mathbf{V}_\mathbf{M}(z)=\begin{pmatrix}1 & 0\\\displaystyle\frac{r(z)\ee^{2\ii t\theta(z;z_0)}}{1-|r(z)|^2} & 1\end{pmatrix}
(1-|r(z)|^2)^{\sigma_3}\begin{pmatrix}1 & \displaystyle -\frac{\overline{r(z)}\ee^{-2\ii t\theta(z;z_0)}}{1-|r(z)|^2}\\0 & 1\end{pmatrix},\quad z<z_0.
\label{eq:factorization-left}
\end{equation}
The importance of these factorizations is that they provide an \emph{algebraic} separation of the oscillatory exponential factors $\ee^{\pm 2\ii t\theta(z;z_0)}$.  Indeed, if the reflection coefficient $r(z)$ is an analytic function of $z\in\mathbb{R}$, then in each case the left-most (right-most) factor has an analytic continuation into the lower (upper) half-plane near the indicated half-line that is exponentially decaying to the identity matrix as $t\to +\infty$ due to $z_0$ being a simple critical point of $\theta(z;z_0)$.  This observation is the basis for the steepest descent method for Riemann-Hilbert problems as first formulated in \cite{Deif6}.  In the more realistic case that $r(z)$ is nowhere analytic, this analytic continuation method must be supplemented with careful approximation arguments that are quite detailed \cite{Deif5}.  We will proceed differently in Section~\ref{sec:extend} below.  But first we need to deal with the central diagonal factor in the factorization \eqref{eq:factorization-left} to be used for $z<z_0$.

\subsection{Modification of diagonal jump}
\label{sec:diagonal}
We now show how the diagonal factor $(1-|r(z)|^2)^{\sigma_3}$ in the jump matrix factorization \eqref{eq:factorization-left} can be replaced with a constant diagonal matrix.  Consider the complex scalar function defined by the formula
\begin{equation}
\delta(z;z_0):=\exp\left(\frac{1}{2\pi\ii}\int_{-\infty}^{z_0}\frac{\ln(1-|r(s)|^2)}{s-z}\,\dd s\right),\quad z\in\mathbb{C}\setminus (-\infty,z_0].
\label{eq:delta-define}
\end{equation}
This function is important because according to the Plemelj formula, it satisfies the scalar jump conditions $\delta_+(z;z_0)=\delta_-(z;z_0)(1-|r(z)|^2)$ for $z<z_0$ and $\delta_+(z;z_0)=\delta_-(z;z_0)$ for $z>z_0$.  Hence the diagonal matrix $\delta(z;z_0)^{\sigma_3}$ is typically used in steepest descent theory to deal with the diagonal factor in \eqref{eq:factorization-left}.  However, $\delta(z;z_0)$ has a mild singularity at $z=z_0$:
\begin{equation}
\delta(z;z_0)=K(z-z_0)^{\ii\nu(z_0)}(1+o(1)),\quad z\to z_0,\quad K=K(z_0)=\text{constant},
\end{equation}
where $\nu(z_0)$ is defined in \eqref{eq:nu-define} and the power function is interpreted as the principal branch.  The use of $\delta(z;z_0)$ introduces this singularity unnecessarily into the Riemann-Hilbert analysis.  In our approach we will therefore use a related function:
\begin{equation}
f(z;z_0):=c(z_0)\delta(z;z_0)(z-z_0)^{-\ii\nu(z_0)},
\label{eq:f-define}
\end{equation}
where the constant $c(z_0)$ is defined by
\begin{equation}
\begin{split}
\quad c(z_0):=&\exp\left(-\frac{1}{2\pi\ii}\left[\int_{-\infty}^{z_0-1}\frac{\ln(1-|r(s)|^2)}{s-z_0}\,\dd s +\int_{z_0-1}^{z_0}\frac{\ln(1-|r(s)|^2)-\ln(1-|r(z_0)|^2)}{s-z_0}\,\dd s\right]\right)\\
=&\exp\left(\frac{1}{2\pi\ii}\int_{-\infty}^{z_0}\ln(z_0-s)\,\dd\ln(1-|r(s)|^2)\right).
\end{split}
\label{eq:c-define}
\end{equation}
The function $f(z;z_0)$ has numerous useful properties that we summarize here.
\begin{lemma}[Properties of $f(z;z_0)$]
Suppose that $r\in H^1(\mathbb{R})$ and there exists $\rho<1$ such that $|r(z)|\le \rho$ holds for all $z\in\mathbb{R}$ (as is implied by $r\in H^{1,1}_1(\mathbb{R})$ which follows from $q_0\in H^{1,1}(\mathbb{R})$).  Then
\begin{itemize}
\item The functions $f(z;z_0)^{\pm 1}$ are well-defined and analytic in $z$ for $\arg(z-z_0)\in (-\pi,\pi)$.
\item The functions $f(z;z_0)^{\pm 1}$ are uniformly bounded independently of $z_0\in\mathbb{R}$:  \begin{equation}
\mathop{\sup_{z_0\in\mathbb{R}}}_{\arg(z-z_0)\in (-\pi,\pi)}|f(z;z_0)|^{\pm 1}\le\frac{1}{1-\rho^2}.
\end{equation}
\item The function $f(z;z_0)$ satisfies the following asymptotic condition:
\begin{equation}
\mathop{\lim_{z\to\infty}}_{-\pi<\arg(z-z_0)<\pi}f(z;z_0)(z-z_0)^{\ii\nu(z_0)}= c(z_0).
\label{eq:f-limit}
\end{equation}
\item The functions $f(z;z_0)^{\pm 2}$ are H\"older continuous with exponent $1/2$.  In particular,  $f(z;z_0)^{\pm 2}\to 1$ as $z\to z_0$ and there is a constant $K=K(\rho)>0$ such that $|f(z;z_0)^{\pm 2}-1|\le K|z-z_0|^{1/2}$ holds whenever $\arg(z-z_0)\in (-\pi,\pi)$.
\item The continuous boundary values $f_\pm(z;z_0)$ taken by $f(z;z_0)$ on $\mathbb{R}$ for $z<z_0$ from $\pm\mathrm{Im}(z)>0$ satisfy the jump condition
\begin{equation}
f_+(z;z_0)=f_-(z;z_0)\frac{1-|r(z)|^2}{1-|r(z_0)|^2},\quad z<z_0.
\end{equation}
\end{itemize}
\label{lem:f}
\end{lemma}
\begin{proof}
The assumptions imply in particular that $\ln(1-|r(\cdot)|^2)\in L^1(\mathbb{R})$, so for $z$ in a small neighborhood of each point disjoint from the integration contour, the integral in \eqref{eq:delta-define} is absolutely convergent and so $\delta(z;z_0)$ and $\delta(z;z_0)^{-1}$ are analytic functions of $z$ on that neighborhood.  The same argument shows that the first integral in the exponent of the expression \eqref{eq:c-define} for $c(z_0)$ is convergent. Since $r\in H^1(\mathbb{R})$ implies that $r(\cdot)$ is H\"older continuous with exponent $1/2$, the condition $|r(\cdot)|\le \rho<1$ further implies that $\ln(1-|r(s)|^2)$ is also H\"older continuous with exponent $1/2$, from which it follows that the second integral in the exponent of the expression \eqref{eq:c-define} is also convergent.  Therefore $c(z_0)$ exists, and clearly $|c(z_0)|=1$.  Since the principal branch of $(z-z_0)^{\mp\ii\nu(z_0)}$ is analytic for $\arg(z-z_0)\in (-\pi,\pi)$, the analyticity of $f(z;z_0)^{\pm 1}$ in the same domain follows.  This proves the first statement.  

In \cite[Proposition 2.12]{Deif5} it is asserted that under the hypothesis $|r(z)|\le\rho<1$, the function $\delta(z;z_0)$ defined by \eqref{eq:delta-define} satisfies the uniform estimates $(1-\rho^2)^{1/2}\le |\delta(z;z_0)|^{\pm 1}\le (1-\rho^2)^{-1/2}$ whenever $\arg(z-z_0)\in (-\pi,\pi)$. 
If $\arg(z-z_0)=0$, then obviously $|\delta(z;z_0)|=1$, so it remains to prove the estimates hold for $\im(z)\neq 0$. 
Following \cite{LiuPS18}, since $\ln(1-\rho^2)\le\ln(1-|r(s)|^2)\le 0$, if $u=\re(z)$ and $v=\im(z)$ we have $\im((s-z)^{-1})=v/((s-u)^2+v^2)$, so assuming $v>0$,
\begin{equation}
 \exp\left(\frac{v\ln(1-\rho^2)}{2\pi}\int_{-\infty}^{z_0}\frac{\dd s}{(s-u)^2+v^2}\right)\le |\delta(u+\ii v;z_0)|.
\end{equation}
Bounding the left-hand side below by extending the integration to $\mathbb{R}$ (using $v\ln(1-\rho^2)<0$) gives the lower bound $(1-\rho^2)^{1/2}\le |\delta(z;z_0)|$, and by taking reciprocals, the upper bound $|\delta(z;z_0)|^{-1}\le (1-\rho^2)^{-1/2}$ for $\im(z)>0$.  
The corresponding result for $\im(z)<0$ follows by the exact symmetry $\delta(\bar{z};z_0)^{-1}=\overline{\delta(z;z_0)}$.
Combining these bounds with $|c(z_0)|=1$ and
the elementary inequalities $(1-\rho^2)^{1/2}\le(1-|r(z_0)|^2)^{1/2}=\ee^{-\pi\nu(z_0)}\le |(z-z_0)^{\ii\nu(z_0)}|\le\ee^{\pi\nu(z_0)}=(1-|r(z_0)|^2)^{-1/2}\le (1-\rho^2)^{-1/2}$ holding for $\arg(z-z_0)\in (-\pi,\pi)$ then proves the second statement.

Since $\ln(1-|r(\cdot)|^2)\in L^1(\mathbb{R})$, from \eqref{eq:delta-define} a dominated convergence argument shows that $\delta(z;z_0)\to 1$ as $z\to\infty$ provided only that the limit is taken in such a way that for some given $\epsilon>0$, $\mathrm{dist}(z,[-\infty,z_0))\ge \epsilon$.
Combining this fact with \eqref{eq:f-define} proves the third statement.

Analyticity implies H\"older continuity, so provided $z$ is bounded away from the half-line $(-\infty,z_0]$, H\"older-$1/2$ continuity of $f(z;z_0)^{\pm 2}$ is obvious.  But, since $\ln(1-|r(\cdot)|^2)$ is H\"older continuous on $\mathbb{R}$ with exponent $1/2$, by the Plemelj-Privalov theorem \cite[\S 19]{Muskhelishvili92} and a related classical result \cite[\S 22]{Muskhelishvili92}, the functions $\delta(z;z_0)^{\pm 1}$ are uniformly H\"older continuous with exponent $1/2$ in any neighborhood of the integration contour \emph{except} for the endpoint $z=z_0$, and hence the same is true for the functions $f(z;z_0)^{\pm 2}$.  However, the latter functions are better-behaved near $z=z_0$. To see this, note that since
\begin{equation}
\begin{split}
(z-z_0)^{\mp\ii\nu(z_0)} &= (z-(z_0-1))^{\mp\ii\nu(z_0)}\left[\frac{z-z_0}{z-(z_0-1)}\right]^{\mp\ii\nu(z_0)}\\ &=(z-(z_0-1))^{\mp\ii\nu(z_0)}\exp\left(\mp\frac{1}{2\pi\ii}\int_{z_0-1}^{z_0}\frac{\ln(1-|r(z_0)|^2)}{s-z}\,\dd s\right),\quad z\in\mathbb{C}\setminus (-\infty,z_0],
\end{split}
\end{equation}
we have from \eqref{eq:delta-define} and \eqref{eq:f-define} that
\begin{multline}
f(z;z_0)^{\pm 2}=\\
c(z_0)^{\pm 2}(z-(z_0-1))^{\mp 2\ii\nu(z_0)}\exp\left(\pm\frac{1}{\pi\ii}\int_{-\infty}^{z_0-1}\frac{\ln(1-|r(s)|^2)}{s-z}\,\dd s\right)
\exp\left(\pm\frac{1}{\pi\ii}\int_{z_0-1}^{+\infty}\frac{h(s)\,\dd s}{s-z}\right)
\label{eq:f-rewrite}
\end{multline}
where $h(s):=\ln(1-|r(s)|^2)-\ln(1-r(z_0)|^2)$ for $s<z_0$ and $h(s):=0$ for $s\ge z_0$.  As the first three factors are analytic at $z=z_0$ while $h(s)$ is H\"older continuous with exponent $1/2$ in a neighborhood of $s=z_0$, the same arguments cited above apply and yield the desired H\"older continuity of $f(z;z_0)^{\pm 2}$ near $z=z_0$.  It only remains to show that $f(z_0;z_0)^{\pm 2}=1$, but this follows immediately from \eqref{eq:c-define} and \eqref{eq:f-rewrite}.  This proves the fourth statement.

Finally, the fifth statement follows from the definition \eqref{eq:f-define} of $f(z;z_0)$ and the jump condition $\delta_+(z;z_0)=\delta_-(z;z_0)(1-|r(z)|^2)$ for $z<z_0$.
\end{proof}

Using the diagonal matrix $f(z;z_0)^{\sigma_3}$ to conjugate the unknown $\mathbf{M}(z)$ of Riemann-Hilbert Problem~\ref{rhp:01} by introducing
\begin{multline}
\mathbf{N}(z)=\mathbf{N}(z;x,t):=\ee^{\ii\omega(z_0)\sigma_3/2}\ee^{\ii t\theta(z_0;z_0)\sigma_3}\cdot c(z_0)^{\sigma_3}\mathbf{M}(z;x,t)f(z;z_0)^{-\sigma_3}\cdot \ee^{-\ii t\theta(z_0;z_0)\sigma_3}\ee^{-\ii\omega(z_0)\sigma_3/2},\\ z\in\mathbb{C}\setminus\Sigma=\mathbb{R},
\label{eq:N-M}
\end{multline}
where 
\begin{equation}
\omega(z_0):=\arg(r(z_0)),
\label{eq:omega-define}
\end{equation}
it is easy to check that $\mathbf{N}(z)$ satisfies several conditions explicitly related to those of $\mathbf{M}(z)$ according to Riemann-Hilbert Problem~\ref{rhp:01}.  Indeed, $\mathbf{N}(z)$ must be a solution of the following equivalent problem. 
\begin{rhp}
Given parameters $(x,t)\in\mathbb{R}^2$, find $\mathbf{N} = \mathbf{N}(z) = \mathbf{N}(z;x,t)$, a $2 \times 2$ matrix, satisfying the following conditions:
\begin{itemize}
\item[]\textit{\textbf{Analyticity:}} $\mathbf{N}$ is an analytic function of $z$ in the domain $\mathbb{C}\setminus\mathbb{R}$.  Moreover, $\mathbf{N}$ has a continuous extension to the real axis from the upper (lower) half-plane denoted $\mathbf{N}_+(z)$ ($\mathbf{N}_-(z)$) for $z\in\mathbb{R}$.
\item[]\textit{\textbf{Jump condition:}} The boundary values satisfy the jump condition
\begin{equation}
\mathbf{N}_+(z)=\mathbf{N}_-(z)\mathbf{V}_\mathbf{N}(z),\quad z\in\mathbb{R},
\end{equation}
where the jump matrix $\mathbf{V}_\mathbf{N}(z)$ may be written in the alternate forms
\begin{multline}
\mathbf{V}_\mathbf{N}(z)=\begin{pmatrix}1 & -f(z;z_0)^{2}\overline{r(z)}\ee^{\ii\omega(z_0)}\ee^{-2\ii t[\theta(z;z_0)-\theta(z_0;z_0)]}\\0 & 1\end{pmatrix}\\{}\cdot
\begin{pmatrix}1 & 0\\f(z;z_0)^{-2}r(z)\ee^{-\ii\omega(z_0)}\ee^{2\ii t[\theta(z;z_0)-\theta(z_0;z_0)]} & 1\end{pmatrix},\quad z>z_0,
\end{multline}
\begin{multline}
\mathbf{V}_\mathbf{N}(z):=\begin{pmatrix}1 & 0 \\\displaystyle \frac{f_-(z;z_0)^{-2}r(z)\ee^{-\ii\omega(z_0)}\ee^{2\ii t[\theta(z;z_0)-\theta(z_0;z_0)]}}{1-|r(z)|^2} & 1\end{pmatrix}\\{}\cdot(1-|r(z_0)|^2)^{\sigma_3}\begin{pmatrix}1 & \displaystyle -\frac{f_+(z;z_0)^{2}\overline{r(z)}\ee^{\ii\omega(z_0)}\ee^{-2\ii t[\theta(z;z_0)-\theta(z_0;z_0)]}}{1-|r(z)|^2}\\0 & 1\end{pmatrix},\quad z<z_0,
\end{multline}
where $f_+(z;z_0)$ ($f_-(z;z_0)$) is the boundary value taken by $f(z;z_0)$ from the upper (lower) half-plane.
\item[]\textit{\textbf{Normalization:}}  There is a matrix $\mathbf{N}_1(x,t)$ such that
\begin{equation}
\mathbf{N}(z)(z-z_0)^{-\ii\nu(z_0)\sigma_3}=\mathbb{I} + z^{-1}\mathbf{N}_1(x,t) + o(z^{-1}),\quad z\to\infty.
\label{eq:N-normalize}
\end{equation}
\end{itemize}
\label{rhp:N}
\end{rhp}
Note that the matrix coefficient $\mathbf{N}_1(x,t)$ is necessarily related to the coefficient $\mathbf{M}_1(x,t)$ in Riemann-Hilbert Problem~\ref{rhp:01} by a diagonal conjugation:  
\begin{equation}
\mathbf{M}_1(x,t)=\ee^{-\ii\omega(z_0)\sigma_3/2}\ee^{-\ii t\theta(z_0;z_0)\sigma_3}c(z_0)^{-\sigma_3}\mathbf{N}_1(x,t)c(z_0)^{\sigma_3}\ee^{\ii t\theta(z_0;z_0)\sigma_3}\ee^{\ii\omega(z_0)\sigma_3/2}.  
\end{equation}
Therefore, the reconstruction formula \eqref{qMrelation} can be written in terms of $\mathbf{N}_1(x,t)$ as
\begin{equation}
q(x,t):=2\ii \ee^{-\ii\omega(z_0)}\ee^{-2\ii t\theta(z_0;z_0)}c(z_0)^{-2}N_{1,12}(x,t).
\label{eq:qNrelation}
\end{equation}
The net effect of this step is therefore to replace the non-constant diagonal central factor in \eqref{eq:factorization-left} with its constant value at $z=z_0$ and to introduce power-law asymptotics at $z=\infty$ at the cost of slight modifications of the left-most and right-most factors in \eqref{eq:factorization-right}--\eqref{eq:factorization-left}.  In the formula \eqref{eq:N-M} we have also taken the opportunity to conjugate off the constant value of $\theta(z;z_0)$ and the phase of $r(z)$ at the critical point $z=z_0$.

\subsection{Nonanalaytic extensions and $\dbar$ steepest descent}
\label{sec:extend}
The key to the steepest descent method, both in its classical analytic framework and in the $\dbar$ setting, is to get the oscillatory factors $\ee^{\pm 2\ii t\theta(z;z_0)}$ off the real axis and into appropriate sectors of the complex $z$-plane where they decay as $t\to+\infty$.  We will accomplish this by exactly the same means as in the linear case, namely by defining non-analytic extensions of the non-oscillatory coefficients of $\ee^{\pm 2\ii t\theta(z;z_0)}$ in the left-most and right-most jump matrix factors in \eqref{eq:factorization-right}--\eqref{eq:factorization-left} by a slight generalization of the formula \eqref{eq:extension-linear}.  In reference to the diagram in Figure~\ref{fig:RHP-deform},
\begin{figure}[h]
\begin{center}
\includegraphics{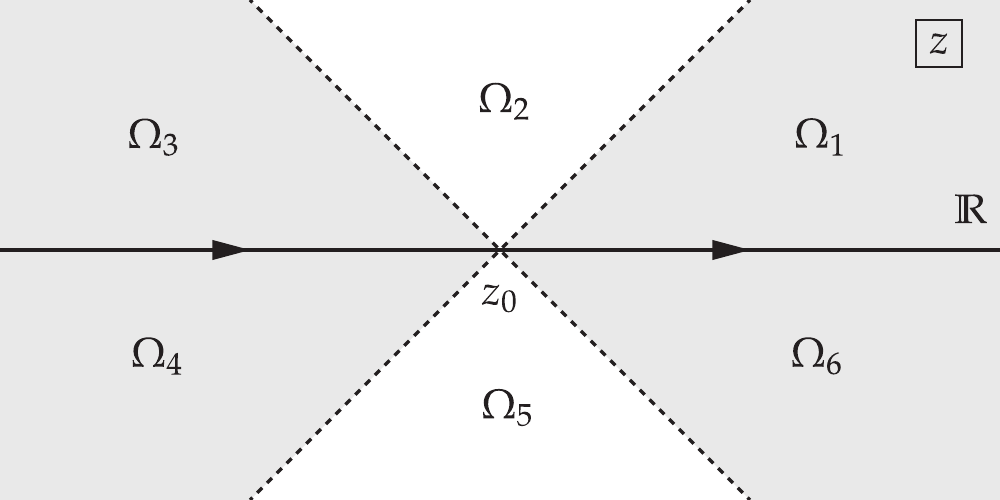}
\end{center}
\caption{The jump contour $\mathbb{R}$ in Riemann-Hilbert Problem~\ref{rhp:N} and the sectors $\Omega_j$, $j=1,\dots,6$ in the $z=u+\ii v$ plane.}
\label{fig:RHP-deform}
\end{figure}
we define sectors 
\begin{equation}
\begin{split}
\Omega_1:&\quad 0<\arg(z-z_0)<\frac{1}{4}\pi\\
\Omega_2:&\quad \frac{1}{4}\pi<\arg(z-z_0)<\frac{3}{4}\pi\\
\Omega_3:&\quad\frac{3}{4}\pi<\arg(z-z_0)<\pi\\
\Omega_4:&\quad-\pi<\arg(z-z_0)<-\frac{3}{4}\pi\\
\Omega_5:&\quad-\frac{3}{4}\pi<\arg(z-z_0)<-\frac{1}{4}\pi\\
\Omega_6:&\quad-\frac{1}{4}\pi<\arg(z-z_0)<0.
\end{split}
\end{equation}
Note that $\Omega_3=\Omega_+$ and $\Omega_6=\Omega_-$ in reference to Figure~\ref{fig:Integral}. Now we define extensions on the domains shaded in Figure~\ref{fig:RHP-deform} by following a very similar approach as in Section~\ref{sec:linear}:
\begin{equation}
\begin{split}
E_1(u,v)&:=\cos(2\arg(u+\ii v-z_0))f(u+\ii v;z_0)^{-2}r(u)\ee^{-\ii\omega(z_0)}
\\
&\quad{}+ (1-\cos(2\arg(u+\ii v-z_0)))|r(z_0)|,\quad z=u+\ii v\in\Omega_1\\
E_3(u,v)&:=
-\left[\cos(2\arg(u+\ii v-z_0))f(u+\ii v;z_0)^2\frac{\overline{r(u)}\ee^{\ii\omega(z_0)}}{1-|r(u)|^2}
\vphantom{+(1-\cos(2\arg(u+\ii v-z_0)))\frac{|r(z_0)|}{1-|r(z_0)|^2}}\right.\\
&\quad{}\left.
\vphantom{\cos(2\arg(u+\ii v-z_0))\frac{\overline{r(u)}\ee^{\ii\omega(z_0)}}{1-|r(u)|^2}}
+(1-\cos(2\arg(u+\ii v-z_0)))\frac{|r(z_0)|}{1-|r(z_0)|^2}\right],\quad z=u+\ii v\in\Omega_3\\
E_4(u,v)&:=
\cos(2\arg(u+\ii v-z_0))f(u+\ii v;z_0)^{-2}\frac{r(u)\ee^{-\ii\omega(z_0)}}{1-|r(u)|^2}
\\
&\quad+(1-\cos(2\arg(u+\ii v-z_0)))\frac{|r(z_0)|}{1-|r(z_0)|^2},\quad z=u+\ii v\in\Omega_4\\
E_6(u,v)&:=
-\left[\cos(2\arg(u+\ii v-z_0))f(u+\ii v;z_0)^2\overline{r(u)}\ee^{\ii\omega(z_0)}
\vphantom{+(1-\cos(2\arg(u+\ii v-z_0)))|r(z_0)|f(u+\ii v;z_0)^{-2}}\right.\\
&\quad{}\left.
\vphantom{\cos(2\arg(u+\ii v-z_0))\overline{r(u)}\ee^{\ii\omega(z_0)}}
+(1-\cos(2\arg(u+\ii v-z_0)))|r(z_0)|\right],\quad z=u+\ii v\in\Omega_6.
\end{split}
\label{eq:dbar-extensions}
\end{equation}
It is easy to check that:
\begin{itemize}
\item $E_1(u,v)$ evaluates to $f(z;z_0)^{-2}r(z)\ee^{-\ii\omega(z_0)}$ for $z\in\mathbb{R}$ on the boundary of $\Omega_1$.
\item $E_3(u,v)$ evaluates to $-f_+(z;z_0)^2\overline{r(z)}\ee^{\ii\omega(z_0)}/(1-|r(z)|^2)$ for $z\in\mathbb{R}$ on the boundary of $\Omega_3$.
\item $E_4(u,v)$ evaluates to $f_-(z;z_0)^{-2}r(z)\ee^{-\ii\omega(z_0)}/(1-|r(z)|^2)$ for $z\in\mathbb{R}$ on the boundary of $\Omega_4$.
\item $E_6(u,v)$ evaluates to $-f(z;z_0)^2\overline{r(z)}\ee^{\ii\omega(z_0)}$ for $z\in\mathbb{R}$ on the boundary of $\Omega_6$.
\end{itemize}
Thus exactly as in Section~\ref{sec:linear} these formul\ae\ represent extensions of their values on the real sector boundaries into the complex plane that become constant on the diagonal sector boundaries, with the constant chosen in each case to ensure continuity of the extension along the interior boundary of each sector.  The only essential difference between the extension formul\ae\ \eqref{eq:dbar-extensions} and the formula \eqref{eq:extension-linear} from Section~\ref{sec:linear} is the way that the factors $f(z;z_0)^{\pm 2}$ are treated differently from the factors involving $r(z)$; the reason for using $f(u+\ii v;z_0)^{\pm 2}$ in \eqref{eq:dbar-extensions} rather than $f(u;z_0)^{\pm 2}$ will become clearer in Section~\ref{sec:dbar} when we compute $\dbar E_j(u,v)$, $j=1,3,4,6$, and take advantage of the fact (see Lemma~\ref{lem:f}) that $\dbar f(u+\ii v;z_0)^{\pm 2}\equiv 0$ in the interior of each sector.

We use the extensions to ``open lenses'' about the intervals $z<z_0$ and $z>z_0$ by making another substitution:
\begin{equation}
\mathbf{O}(u,v;x,t):=\begin{cases}
\mathbf{N}(z;x,t)\begin{pmatrix}1&0\\E_1(u,v)\ee^{2\ii t[\theta(u+\ii v;z_0)-\theta(z_0;z_0)]} & 1\end{pmatrix}^{-1},&\quad z=u+\ii v\in\Omega_1\\
\mathbf{N}(z;x,t),&\quad z=u+\ii v\in\Omega_2\\
\mathbf{N}(z;x,t)\begin{pmatrix}1&E_3(u,v)\ee^{-2\ii t[\theta(u+\ii v;z_0)-\theta(z_0;z_0)]} \\0 & 1\end{pmatrix}^{-1},&\quad z=u+\ii v\in\Omega_3\\
\mathbf{N}(z;x,t)\begin{pmatrix}1&0\\E_4(u,v)\ee^{2\ii t[\theta(u+\ii v;z_0)-\theta(z_0;z_0)]} & 1\end{pmatrix},&\quad z=u+\ii v\in\Omega_4\\
\mathbf{N}(z;x,t),&\quad z=u+\ii v\in\Omega_5\\
\mathbf{N}(z;x,t)\begin{pmatrix}1&E_6(u,v)\ee^{-2\ii t[\theta(u+\ii v;z_0)-\theta(z_0;z_0)]} \\ 0 & 1\end{pmatrix},&\quad z=u+\ii v\in\Omega_6.
\end{cases}
\label{eq:O-N}
\end{equation}
Our notation $\mathbf{O}(u,v;x,t)$ reflects the viewpoint that unlike $\mathbf{N}(z;x,t)$, $z=u+\ii v$, $\mathbf{O}(u,v;x,t)$ is not a piecewise-analytic function in the complex plane due to the non-analytic extensions $E_j(u,v)$, $j=1,3,4,6$.
The exponential factors in \eqref{eq:O-N} all have modulus less than $1$ and decay exponentially to zero as $t\to+\infty$ pointwise in the interior of each of the indicated sectors, a fact that suggests that \eqref{eq:O-N} is a near-identity transformation in the limit $t\to+\infty$. We also have the following property.
\begin{lemma}[Relation between $\mathbf{N}$ and $\mathbf{O}$ for large $z\in\mathbb{C}$]  
Let $z_0\in\mathbb{R}$ be fixed, and suppose that $r\in H^1(\mathbb{R})$ and that there exists a constant $\rho<1$ such that $|r(z)|\le \rho$ holds for all $z\in\mathbb{R}$ (conditions that are true for $r\in H^{1,1}_1(\mathbb{R})$ as follows from $q_0\in H^{1,1}(\mathbb{R})$).  Then
$\mathbf{O}(u,v;x,t)=\mathbf{N}(u+\ii v;x,t)(\mathbb{I}+o(1))$ holds as $z=u+\ii v\to\infty$ where the decay of the error term is uniform with respect to direction in each sector $\Omega_j$, $j=1,\dots,6$.  
\label{lem:decay}
\end{lemma}
\begin{proof}
The exponential factors in \eqref{eq:O-N} also decay as $z=u+\ii v\to\infty$ provided that $v\to\infty$.  
Since $r,r'\in L^2(\mathbb{R})$ means that $(1+|\cdot|)\hat{r}(\cdot)$ is square-integrable where $\hat{r}$ denotes the Fourier transform of $r$, the Cauchy-Schwarz inequality implies that also $\hat{r}\in L^1(\mathbb{R})$.  Hence by the Riemann-Lebesgue Lemma, $r(u)$ is bounded, continuous, and tends to zero as $u\to\infty$.  As $1-|r(u)|^2\ge 1-\rho^2>0$, the same properties hold for $r(u)/(1-|r(u)|^2)$.  Since the hypotheses of Lemma~\ref{lem:f} hold, $f(u+\ii v;z_0)^{\pm 2}$ are bounded functions, so the desired result follows from using extension formul\ae\ \eqref{eq:dbar-extensions} in \eqref{eq:O-N}.
\end{proof}
Despite the non-analyticity of the extensions, the above proof shows also that each of the extensions $E_j(u,v)$, $j=1,3,4,6$, is continuous on the relevant sector and therefore $\mathbf{O}(u,v;x,t)$ is a piecewise-continuous function of $(u,v)\in\mathbb{R}^2$ with jump discontinuities across the sector boundaries.  We address these jump discontinuities next.

\subsection{The isomonodromy problem of Its}
\label{sec:PC}
Although $\mathbf{O}(u,v;x,t)$ is not analytic in the sectors shaded in Figure~\ref{fig:RHP-deform} for essentially the same reason that the double integral error term in \eqref{eq:q-Stokes} does not vanish identically, the fact that the extensions $E_j(u,v)$, $j=1,3,4,6$, evaluate to constants on the diagonals:
\begin{equation}
\begin{split}
E_1(u-z_0,u)=|r(z_0)|\quad&\text{and}\quad E_6(u-z_0,-u)=-|r(z_0)|,\quad u>z_0,\\
E_3(u-z_0,-u)=-\frac{|r(z_0)|}{1-|r(z_0)|^2}\quad&\text{and}\quad E_4(u-z_0,u)=\frac{|r(z_0)|}{1-|r(z_0)|^2},\quad u<z_0,
\end{split}
\end{equation}
implies that if we introduce the recentered and rescaled independent variable $\zeta:=2t^{1/2}(z-z_0)$, the jump conditions satisfied by $\mathbf{O}(u,v;x,t)$ across the sector boundaries are exactly the same as those satisfied by the matrix function $\mathbf{P}(\zeta;|r(z_0)|)$ solving the following Riemann-Hilbert problem.
\begin{rhp}
Let $m\in [0,1)$ be a parameter, and seek a $2\times 2$ matrix function $\mathbf{P}=\mathbf{P}(\zeta)=\mathbf{P}(\zeta;m)$ with the following properties:
\begin{itemize}
\item[]\textit{\textbf{Analyticity:}} $\mathbf{P}(\zeta)$ is an analytic function of $\zeta$ in the sectors $|\arg(\zeta)|<\tfrac{1}{4}\pi$, $\tfrac{1}{4}\pi<\pm\arg(\zeta)<\tfrac{3}{4}\pi$, and $\tfrac{3}{4}\pi<\pm\arg(\zeta)<\pi$.  It admits a continuous extension from each of these five sectors to its boundary.
\item[]\textit{\textbf{Jump conditions:}} Denoting by $\mathbf{P}_+(\zeta)$ (resp., $\mathbf{P}_-(\zeta)$) the boundary value taken on any one of the rays of the jump contour $\Sigma_\mathbf{P}$ from the left (resp., right) according to the orientation shown in Figure~\ref{fig:PC-jumps}, the boundary values are related by $\mathbf{P}_+(\zeta;m)=\mathbf{P}_-(\zeta;m)\mathbf{V}_\mathbf{P}(\zeta;m)$, where the jump matrix $\mathbf{V}_\mathbf{P}(\zeta;m)$ is defined on the five rays of $\Sigma_\mathbf{P}$ by
\begin{figure}[h]
\begin{center}
\includegraphics{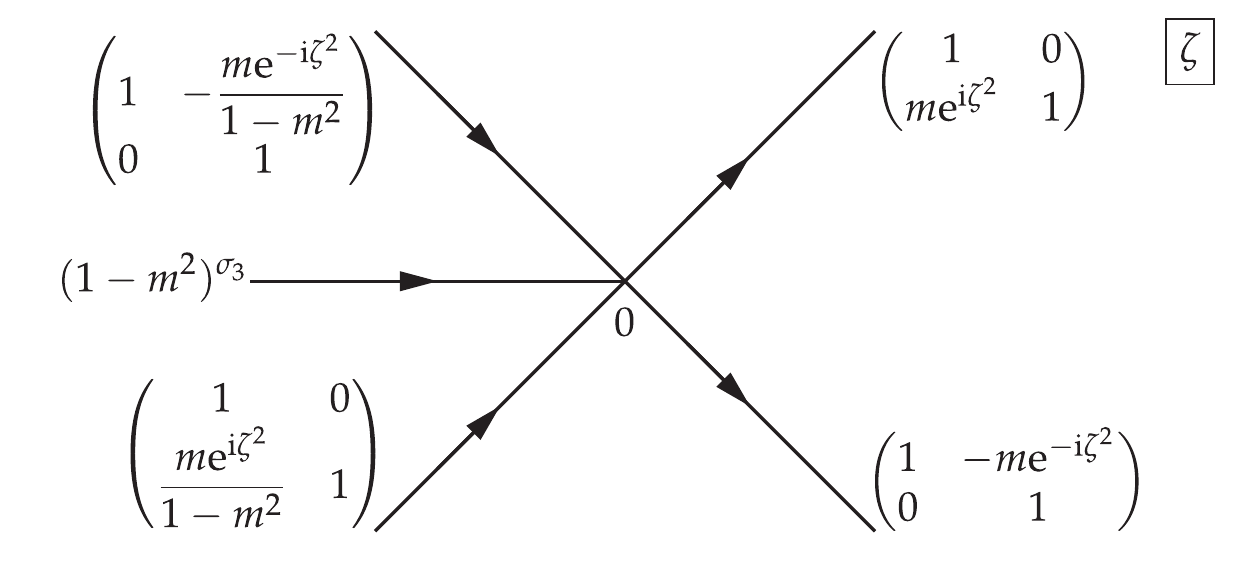}
\end{center}
\caption{The jump contour $\Sigma_\mathbf{P}$ and jump matrix $\mathbf{V}_\mathbf{P}(\zeta;m)$ for Riemann-Hilbert Problem~\ref{rhp:PC}.}
\label{fig:PC-jumps}
\end{figure}
\begin{equation}
\mathbf{V}_\mathbf{P}(\zeta;m):=\begin{cases}
\displaystyle \begin{pmatrix}1&0\\m\ee^{\ii\zeta^2} & 1\end{pmatrix},&\quad \arg(\zeta)=\frac{1}{4}\pi\\
\displaystyle \begin{pmatrix}1&-m\ee^{-\ii\zeta^2}\\0 & 1\end{pmatrix},&\quad\arg(\zeta)=-\frac{1}{4}\pi\\
\displaystyle \begin{pmatrix}1 & \displaystyle -\frac{m\ee^{-\ii\zeta^2}}{1-m^2}\\0 & 1\end{pmatrix},&\quad\arg(\zeta)=\frac{3}{4}\pi\\
\displaystyle\begin{pmatrix}1 & 0\\\displaystyle\frac{m\ee^{\ii\zeta^2}}{1-m^2}&1\end{pmatrix},&\quad\arg(\zeta)=-\frac{3}{4}\pi\\
(1-m^2)^{\sigma_3},&\quad\arg(-\zeta)=0.
\end{cases}
\label{eq:P-jumps}
\end{equation}
\item[]\textit{\textbf{Normalization:}} $\mathbf{P}(\zeta;m)\zeta^{-\ln(1-m^2)\sigma_3/(2\pi\ii)}\to\mathbb{I}$ as $\zeta\to\infty$.
\end{itemize}
\label{rhp:PC}
\end{rhp}
This Riemann-Hilbert problem is essentially the isomonodromy problem identified by Its \cite{Its81}, and it is the analogue in the nonlinear setting of the Gaussian integral that is the leading term of the stationary phase expansion \eqref{eq:q-Stokes} in the linear case.  Although the jump conditions for $\mathbf{O}(u,v;x,t)$ correspond exactly to those of $\mathbf{P}(\zeta;|r(z_0)|)$, the scaling $z\mapsto \zeta=2t^{1/2}(z-z_0)$ introduces an extra factor into the asymptotics as $z\to\infty$; the fact of the matter is that
the matrix $(2t^{1/2})^{\ii\nu(z_0)\sigma_3}\mathbf{O}(u,v;x,t)$ satisfies the normalization condition of $\mathbf{P}(\zeta;|r(z_0)|)$, and the constant pre-factor has no effect on the jump conditions.  Hence in Section~\ref{sec:dbar} below we shall use the latter as a parametrix for the former. 

However, we first develop the explicit solution of Riemann-Hilbert Problem~\ref{rhp:PC}.  The first step is to consider the related unknown $\mathbf{U}(\zeta;m):=\mathbf{P}(\zeta;m)\ee^{-\ii\zeta^2\sigma_3/2}$ and observe that from the conditions of Riemann-Hilbert Problem~\ref{rhp:PC} that $\mathbf{U}(\zeta;m)$ is analytic exactly in the same five sectors where $\mathbf{P}(\zeta;m)$ is, and that it satisfies jump conditions of exactly the form \eqref{eq:P-jumps} except that the factors $\ee^{\pm\ii\zeta^2}$ are everywhere replaced by $1$; in other words, the jump matrix for $\mathbf{U}(\zeta;m)$ on each jump ray is constant along the ray.  It follows that the $\zeta$-derivative $\mathbf{U}'(\zeta;m)$ satisfies the same ``raywise constant'' jump conditions as does $\mathbf{U}(\zeta;m)$ itself.  Then, since it is easy to prove by Liouville's theorem that any solution $\mathbf{P}(\zeta;m)$ of Riemann-Hilbert Problem~\ref{rhp:PC} has unit determinant, it follows that $\mathbf{U}(\zeta;m)$ is invertible and a calculation shows that the function $\mathbf{U}'(\zeta;m)\mathbf{U}(\zeta;m)^{-1}$ is continuous and hence by Morera's theorem analytic in the whole $\zeta$-plane possibly excepting $\zeta=0$.  We will assume analyticity at the origin as well and show later that this is consistent.  As an entire function of $\zeta$, the product $\mathbf{U}'(\zeta;m)\mathbf{U}(\zeta;m)^{-1}$ is potentially determined by its asymptotic behavior as $\zeta\to\infty$.  Assuming further that the normalization condition in Riemann-Hilbert Problem~\ref{rhp:PC} means both that for some matrix coefficient $\mathbf{P}_1(m)$ to be determined,
\begin{equation}
\begin{split}
\mathbf{P}(\zeta;m)&=\left(\mathbb{I} + \zeta^{-1}\mathbf{P}_1(m) + \mathcal{O}(\zeta^{-2})\right)\zeta^{\ln(1-m^2)\sigma_3/(2\pi\ii)}\quad\text{and}\\
\mathbf{P}'(\zeta;m)&=\left(\frac{\ln(1-m^2)}{2\pi\ii}\zeta^{-1}\sigma_3 +\mathcal{O}(\zeta^{-2})\right)\zeta^{\ln(1-m^2)\sigma_3/(2\pi\ii)}
\end{split}
\label{eq:P-expansion}
\end{equation}
hold as $\zeta\to\infty$, such as would arise from term-by-term differentiation, it follows also that
\begin{equation}
\begin{split}
\mathbf{U}(\zeta;m)&=\left(\mathbb{I}+\zeta^{-1}\mathbf{P}_1(m)+\mathcal{O}(\zeta^{-2})\right)\zeta^{\ln(1-m^2)\sigma_3/(2\pi\ii)}\ee^{-\ii\zeta^2\sigma_3/2}\quad\text{and}\\
\mathbf{U}'(\zeta;m)&=\left(-\ii\zeta\sigma_3 -\ii\mathbf{P}_1(m)\sigma_3 +\mathcal{O}(\zeta^{-1})\right)\zeta^{\ln(1-m^2)\sigma_3/(2\pi\ii)}\ee^{-\ii\zeta^2\sigma_3/2}
\end{split}
\label{eq:U-expansion}
\end{equation}
as $\zeta\to\infty$.  Therefore the entire function is determined by Liouville's theorem to be a linear polynomial:
\begin{equation}
\mathbf{U}'(\zeta;m)\mathbf{U}(\zeta;m)^{-1}=-\ii\zeta\sigma_3 +\ii[\sigma_3,\mathbf{P}_1(m)],
\end{equation}
where $[\mathbf{A},\mathbf{B}]:=\mathbf{AB}-\mathbf{BA}$ is the matrix commutator.  In other words, $\mathbf{U}(\zeta;m)$ satisfies the first-order system of linear differential equations:
\begin{equation}
\frac{\dd\mathbf{U}}{\dd\zeta}(\zeta;m)=\begin{pmatrix}-\ii\zeta & 2\ii P_{1,12}(m)\\-2\ii P_{1,21}(m) & \ii\zeta\end{pmatrix}\mathbf{U}(\zeta;m).
\label{eq:PC-system}
\end{equation}
Now, another easy consequence of Liouville's theorem is that there is at most one solution of Riemann-Hilbert Problem~\ref{rhp:PC}.  Using the fact that $m\in [0,1)$, it is not difficult to show that if $\mathbf{P}(\zeta;m)$ is a solution of Riemann-Hilbert Problem~\ref{rhp:PC}, then so is 
\begin{equation}
\sigma_1\overline{\mathbf{P}(\overline{\zeta};m)}\sigma_1,\quad\text{where}\quad
\sigma_1:=\begin{pmatrix}0&1\\1&0\end{pmatrix},
\end{equation}
so by uniqueness it follows that $\mathbf{P}(\zeta;m)=\sigma_1\overline{\mathbf{P}(\overline{\zeta};m)}\sigma_1$.  Combining this symmetry with the first expansion in \eqref{eq:P-expansion} shows that $P_{1,21}(m)=\overline{P_{1,12}(m)}$, so the differential equations can be written in the form
\begin{equation}
\frac{\dd\mathbf{U}}{\dd\zeta}(\zeta;m)=\begin{pmatrix}-\ii\zeta & \beta\\\overline{\beta} & \ii\zeta\end{pmatrix}\mathbf{U}(\zeta;m),\quad\beta=\beta(m):=2\ii P_{1,12}(m).
\label{eq:PC-system-2}
\end{equation}
The constant $\beta\in\mathbb{C}$ is unknown, but if it is considered as a parameter, then eliminating the second row shows that the elements $U_{1j}$, $j=1,2$, of the first row satisfy Weber's equation for parabolic cylinder functions in the form:
\begin{equation}
\frac{\dd^2 U_{1j}}{\dd y^2}-\left(\frac{1}{4}y^2+a\right)U_{1j}=0,\quad a:=\frac{1}{2}(1+\ii|\beta|^2),\quad y:=\sqrt{2}\ee^{-\ii\pi/4}\zeta,\quad j=1,2.
\label{eq:Weber}
\end{equation}
The solutions of this equation are well-documented in the Digital Library of Mathematical Functions \cite[\S12]{DLMF}.  The equation \eqref{eq:Weber} has  particular solutions denoted $U(a,\pm y)$ and $U(-a,\pm\ii y)$, where $U(\cdot,\cdot)$ is a special function\footnote{In many works on long-time asymptotics for the Cauchy problem \eqref{eq:NLSEQ}--\eqref{eq:IC} written before the Digital Library of Mathematical Functions was freely available (\emph{e.g.}, \cite{Deif5,DiengM08}), the solution of Riemann-Hilbert Problem~\ref{rhp:PC} was developed in terms of the related function $D_\nu(y):=U(-\tfrac{1}{2}-\nu,y)$. Since most formul\ae\ in \cite[\S12]{DLMF} are phrased in terms of $U(\cdot,\cdot)$, we favor the latter.} with well-known integral representations, asymptotic expansions, and connection formul\ae.

The second step is to represent the elements $U_{1j}$ as linear combinations of a fundamental pair of so-called numerically satisfactory solutions specially adapted to each of the five sectors of analyticity for Riemann-Hilbert Problem~\ref{rhp:PC}.  Thus, we write
\begin{equation}
U_{1j}(\zeta;m)=\begin{cases}
\beta A_j^{(0)}U(a,y) + \beta B_j^{(0)}U(-a,\ii y),&\quad|\arg(\zeta)|<\frac{1}{4}\pi,\\
\beta A_j^{(1)}U(a,y) + \beta B_j^{(1)}U(-a,-\ii y),&\quad\frac{1}{4}\pi<\arg(\zeta)<\frac{3}{4}\pi,\\
\beta A_j^{(-1)}U(a,-y)+\beta B_j^{(-1)}U(-a,\ii y),&\quad-\frac{3}{4}\pi<\arg(\zeta)<-\frac{1}{4}\pi,\\
\beta A_j^{(2)}U(a,-y)+\beta B_j^{(2)}U(-a,-\ii y),&\quad \frac{3}{4}\pi<\arg(\zeta)<\pi,\\
\beta A_j^{(-2)}U(a,-y)+\beta B_j^{(-2)}U(-a,-\ii y),&\quad -\pi<\arg(\zeta)<-\frac{3}{4}\pi.
\end{cases}
\label{eq:U-first-row}
\end{equation}
and then using the first row of \eqref{eq:PC-system-2} along with identities allowing the elimination of derivatives of $U$ \cite[Eqs.\@ 12.8.2--12.8.3]{DLMF} we get the following representation of the elements of the second row of $\mathbf{U}(\zeta;m)$:
\begin{equation}
U_{2j}(\zeta;m)=\sqrt{2}\ee^{-\ii\pi/4}\begin{cases}
-A_j^{(0)}U(a-1,y)+\ii(a-\tfrac{1}{2})B_j^{(0)}U(1-a,\ii y),&\quad |\arg(\zeta)|<\frac{1}{4}\pi,\\
-A_j^{(1)}U(a-1,y)-\ii(a-\tfrac{1}{2})B_j^{(1)}U(1-a,-\ii y),&\quad \frac{1}{4}\pi<\arg(\zeta)<\frac{3}{4}\pi,\\
A_j^{(-1)}U(a-1,-y)+\ii(a-\tfrac{1}{2})B_j^{(-1)}U(1-a,\ii y),&\quad -\frac{3}{4}\pi<\arg(\zeta)<-\frac{1}{4}\pi,\\
A_j^{(2)}U(a-1,-y)-\ii (a-\tfrac{1}{2})B_j^{(2)}U(1-a,-\ii y),&\quad \frac{3}{4}\pi<\arg(\zeta)<\pi,\\
A_j^{(-2)}U(a-1,-y)-\ii(a-\tfrac{1}{2})B_j^{(-2)}U(1-a,-\ii y),&\quad -\pi<\arg(\zeta)<-\frac{3}{4}\pi.
\end{cases}
\label{eq:U-second-row}
\end{equation}

Finally, we determine the coefficients $A_j^{(i)}$ and $B_j^{(i)}$ for $j=1,2$ and $i=0,\pm 1,\pm 2$, as well as the value of $\beta=\beta(m)$ so that all of the conditions of Riemann-Hilbert Problem~\ref{rhp:PC} are satisfied by $\mathbf{P}(\zeta;m)=\mathbf{U}(\zeta;m)\ee^{\ii\zeta\sigma_3/2}$.
The advantage of using numerically satisfactory fundamental pairs is that the asymptotic expansion \cite[Eq.\@ 12.9.1]{DLMF}
\begin{equation}
U(a,y)\sim\ee^{-\tfrac{1}{4}y^2}y^{-a-\tfrac{1}{2}}\sum_{k=0}^\infty (-1)^k\frac{\left(\tfrac{1}{2}+a\right)_{2k}}{k!(2y^2)^k},\quad y\to\infty,\quad |\arg(y)|<\frac{3}{4}\pi
\label{eq:PC-expansion}
\end{equation}
can be used to determine from \eqref{eq:U-first-row}--\eqref{eq:U-second-row} the asymptotic behavior of $\mathbf{U}(\zeta;m)$ in each sector for the purposes of comparison with the first formula in \eqref{eq:U-expansion}.  This immediately shows that for consistency it is necessary to take $A_1^{(i)}=0$ and $B_2^{(i)}=0$ for $i=0,\pm 1,\pm 2$.  Next, it is useful to consider the trivial jump conditions for the first column of $\mathbf{U}(\zeta;m)$ (across $\arg(\zeta)=-\tfrac{1}{4}\pi$ and $\arg(\zeta)=\tfrac{3}{4}\pi$) and for the second column of $\mathbf{U}(\zeta;m)$ (across $\arg(\zeta)=\tfrac{1}{4}\pi$ and $\arg(\zeta)=-\tfrac{3}{4}\pi$).  These imply the identities $B_1^{(0)}=B_1^{(-1)}$, $B_1^{(1)}=B_1^{(2)}$ (from matching the first column) and $A_2^{(0)}=A_2^{(1)}$, $A_2^{(-2)}=A_2^{(-1)}$ (from matching the second column).  The diagonal jump condition satisfied by $\mathbf{U}(\zeta;m)$ across the negative real axis then yields the additional identities $B_1^{(-2)}=(1-m^2)^{-1}B_1^{(2)}$ and $A_2^{(2)}=(1-m^2)^{-1}A_2^{(-2)}$.  With this information, we have found that $\mathbf{U}(\zeta;m)$ necessarily has the form
\begin{equation}
\mathbf{U}(\zeta;m)=\begin{pmatrix}
\beta B_1^{(0)}U(-a,\ii y) & \beta A_2^{(0)}U(a,y)\\
\sqrt{2}\ee^{\ii\pi/4}(a-\tfrac{1}{2})B_1^{(0)}U(1-a,\ii y) & \sqrt{2}\ee^{3\pi\ii/4}A_2^{(0)}U(a-1,y)
\end{pmatrix},\quad |\arg(\zeta)|<\frac{1}{4}\pi,
\label{eq:U-sector-0}
\end{equation}
\begin{equation}
\mathbf{U}(\zeta;m)=\begin{pmatrix}
\beta B_1^{(1)}U(-a,-\ii y) & \beta A_2^{(0)}U(a,y)\\
\sqrt{2}\ee^{-3\pi\ii/4}(a-\tfrac{1}{2})B_1^{(1)}U(1-a,-\ii y) & \sqrt{2}\ee^{3\pi\ii/4}A_2^{(0)}U(a-1,y)
\end{pmatrix},\quad \frac{1}{4}\pi<\arg(\zeta)<\frac{3}{4}\pi,
\end{equation}
\begin{multline}
\mathbf{U}(\zeta;m)=\begin{pmatrix}
\beta B_1^{(0)} U(-a,\ii y) & \beta A_2^{(-1)}U(a,-y)\\
\sqrt{2}\ee^{\ii\pi/4}(a-\tfrac{1}{2})B_1^{(0)}U(1-a,\ii y) & \sqrt{2}\ee^{-\ii\pi/4}A_2^{(-1)}U(a-1,-y)
\end{pmatrix},\\ -\frac{3}{4}\pi<\arg(\zeta)<-\frac{1}{4}\pi,
\end{multline}
\begin{multline}
\mathbf{U}(\zeta;m)=\begin{pmatrix}
\beta B_1^{(1)} U(-a,-\ii y) & \beta(1-m^2)^{-1}A_2^{(-1)}U(a,-y)\\
\sqrt{2}\ee^{-3\pi\ii/4}(a-\tfrac{1}{2})B_1^{(1)}U(1-a,-\ii y) &
\sqrt{2}\ee^{-\ii\pi/4}(1-m^2)^{-1}A_2^{(-1)}U(a-1,-y)
\end{pmatrix},\\ \frac{3}{4}\pi<\arg(\zeta)<\pi,
\end{multline}
and
\begin{multline}
\mathbf{U}(\zeta;m)=\begin{pmatrix}
\beta (1-m^2)^{-1}B_1^{(1)}U(-a,-\ii y) & \beta A_2^{(-1)}U(a,-y)\\
\sqrt{2}\ee^{-3\pi\ii/4}(a-\tfrac{1}{2})(1-m^2)^{-1}B_1^{(1)}U(1-a,-\ii y) & 
\sqrt{2}\ee^{-\ii\pi/4}A_2^{(-1)}U(a-1,-y)
\end{pmatrix},\\
-\pi<\arg(\zeta)<-\frac{3}{4}\pi.
\label{eq:U-sector-minus-2}
\end{multline}
Appealing again to \eqref{eq:PC-expansion} now shows that $\mathbf{U}(\zeta;m)$ agrees with the first formula in \eqref{eq:U-expansion} up to the leading term only if the parameter $a$ in Weber's equation \eqref{eq:Weber} satisfies
\begin{equation}
a-\frac{1}{2}=\frac{1}{2\pi\ii}\ln(1-m^2)\quad\implies\quad |\beta|^2=-\frac{1}{\pi}\ln(1-m^2)>0,
\label{eq:beta-mod-squared}
\end{equation}
and the remaining
constants $A_2^{(0)}$, $A_2^{(-1)}$, $B_1^{(0)}$, and $B_1^{(1)}$, are given in terms of $\beta$ by
\begin{equation}
\begin{split}
B_1^{(0)}&=\beta^{-1}(1-m^2)^{-1/8}\exp\left(\ii\frac{1}{4\pi}\ln(2)\ln(1-m^2)\right)\\
A_2^{(0)}&=\frac{1}{\sqrt{2}}(1-m^2)^{-1/8}\ee^{-3\pi\ii/4}\exp\left(-\ii\frac{1}{4\pi}\ln(2)\ln(1-m^2)\right)\\
B_1^{(1)}&=\beta^{-1}(1-m^2)^{3/8}\exp\left(\ii\frac{1}{4\pi}\ln(2)\ln(1-m^2)\right)\\
A_{2}^{(-1)}&=\frac{1}{\sqrt{2}}(1-m^2)^{3/8}\ee^{\ii\pi/4}\exp\left(-\ii\frac{1}{4\pi}\ln(2)\ln(1-m^2)\right).
\end{split}
\label{eq:remaining-coefficients}
\end{equation}
Only $\arg(\beta)$ remains to be determined, and for this we recall the nontrivial jump conditions for the first (second) column of $\mathbf{U}(\zeta;m)$ across the rays $\arg(\zeta)=\tfrac{1}{4}\pi,-\tfrac{3}{4}\pi$ (the rays $\arg(\zeta)=-\tfrac{1}{4}\pi,\tfrac{3}{4}\pi$).  Actually all four of these jump conditions contain equivalent information due to the fact that the cyclic product of the jump matrices in Riemann-Hilbert Problem~\ref{rhp:PC} about the origin is the identity, so we just examine the transition of the first column across the ray $\arg(\zeta)=\tfrac{1}{4}\pi$ implied by the jump conditions in Riemann-Hilbert Problem~\ref{rhp:PC}.  Using all available information, the jump condition matches the connection formula \cite[Eq.\@ 12.2.18]{DLMF} if and only if
\begin{equation}
\arg(\beta)=\frac{\pi}{4}+\frac{1}{2\pi}\ln(2)\ln(1-m^2)-\arg\left(\Gamma\left(\ii\frac{1}{2\pi}\ln(1-m^2)\right)\right).
\label{eq:arg-beta}
\end{equation}
Combining this with \eqref{eq:beta-mod-squared} determines $\beta=\beta(m)$ and then using \eqref{eq:remaining-coefficients} in \eqref{eq:U-sector-0}--\eqref{eq:U-sector-minus-2} fully determines
$\mathbf{U}(\zeta;m)$ and hence also $\mathbf{P}(\zeta;m)=\mathbf{U}(\zeta;m)\ee^{\ii\zeta^2\sigma_3/2}$.
This completes the construction of the necessarily unique solution of Riemann-Hilbert Problem~\ref{rhp:PC}.  One can easily check directly that $\mathbf{U}'(\zeta;m)\mathbf{U}(\zeta;m)^{-1}$ is analytic at $\zeta=0$, and using \eqref{eq:PC-expansion} (which is known to be a formally differentiable expansion) one confirms the asymptotic expansions \eqref{eq:P-expansion}--\eqref{eq:U-expansion}, justifying after the fact all assumptions made to arrive at the explicit solution.  

We note that for each $m\in [0,1)$, $\mathbf{P}(\zeta;m)$ is uniformly bounded with respect to $\zeta\in\mathbb{C}$, since it is locally bounded and the normalization factor in the asymptotics as $\zeta\to\infty$ satisfies 
\begin{equation}
(1-m^2)^{1/2}<|\zeta^{-\ln(1-m^2)/(2\pi \ii)}|<(1-m^2)^{-1/2},\quad\arg(\zeta)\in (-\pi,\pi).  
\end{equation}
Since $\det(\mathbf{P}(\zeta;m))=1$, the same holds for $\mathbf{P}(\zeta;m)^{-1}$.  Moreover, it is not difficult to see that if $\|\cdot\|$ is a matrix norm, then $\sup_{\zeta\in\mathbb{C}\setminus\Sigma_\mathbf{P}}\|\mathbf{P}(\zeta;m)\|$ is a continuous function of $m\in [0,1)$.  Therefore the estimates on $\mathbf{P}(\zeta;m)$ and $\mathbf{P}(\zeta;m)^{-1}$ hold uniformly with respect to $m\in [0,\rho]$ for any $\rho<1$.  

\subsection{The equivalent $\dbar$ problem and its solution for large $t$}
\label{sec:dbar}
The next part of the proof of Theorem~\ref{mainresult} is the nonlinear analogue of the estimation of the error $\mathcal{E}(x,t)$ in the stationary phase formula \eqref{eq:stationary-phase} by double integrals in the $z$-plane.  Here instead of a double integral we will have a double-integral equation arising from a $\dbar$-problem.  To arrive at this problem, we simply define a matrix function $\mathbf{E}(u,v;x,t)$ by comparing the ``open lenses'' matrix $(2t^{1/2})^{\ii\nu(z_0)\sigma_3}\mathbf{O}(u,v;x,t)$ with its parametrix $\mathbf{P}(2t^{1/2}(z-z_0);|r(z_0)|)$:
\begin{equation}
\mathbf{E}(u,v;x,t):=(2t^{1/2})^{\ii\nu(z_0)\sigma_3}\mathbf{O}(u,v;x,t)\mathbf{P}(2t^{1/2}(u+\ii v-z_0);|r(z_0)|)^{-1}.
\label{eq:E-define}
\end{equation}
We claim that $\mathbf{E}(u,v;x,t)$ satisfies the following problem.
\begin{dbarproblem}
Let $(x,t)\in\mathbb{R}^2$ be parameters.  Find a $2\times 2$ matrix function $\mathbf{E}=\mathbf{E}(u,v)=\mathbf{E}(u,v;x,t)$, $(u,v)\in\mathbb{R}^2$ with the following properties:
\begin{itemize}
\item[]\textit{\textbf{Continuity:}} $\mathbf{E}$ is a continuous function of $(u,v)\in\mathbb{R}^2$.
\item[]\textit{\textbf{Nonanalyticity:}} $\mathbf{E}$ is a (weak) solution of the partial differential equation $\dbar\mathbf{E}(u,v)=\mathbf{E}(u,v)\mathbf{W}(u,v)$, where $\mathbf{W}(u,v)=\mathbf{W}(u,v;x,t)$ is defined by 
\begin{equation}
\mathbf{W}(u,v;x,t):=\mathbf{P}(2t^{1/2}(u+\ii v-z_0);|r(z_0)|)\boldsymbol{\Delta}(u,v;x,t)\mathbf{P}(2t^{1/2}(u+\ii v-z_0);|r(z_0)|)^{-1},
\label{eq:W-def}
\end{equation}
and
\begin{equation}
\boldsymbol{\Delta}(u,v;x,t):=\begin{cases}
\begin{pmatrix}0&0\\-\dbar E_1(u,v)\ee^{2\ii t(\theta(u+\ii v;z_0)-\theta(z_0;z_0))} &0\end{pmatrix}
,&\quad u+\ii v\in\Omega_1\\
\mathbf{0},&\quad u+\ii v\in\Omega_2\\
\begin{pmatrix}0&-\dbar E_3(u,v)\ee^{-2\ii t(\theta(u+\ii v;z_0)-\theta(z_0;z_0))}\\ 0&0\end{pmatrix}
,&\quad u+\ii v\in\Omega_3\\
\begin{pmatrix}0&0\\\dbar E_4(u,v)\ee^{2\ii t(\theta(u+\ii v;z_0)-\theta(z_0;z_0))} &0\end{pmatrix}
,&\quad u+\ii v\in\Omega_4\\
\mathbf{0},&\quad u+\ii v\in\Omega_5\\
\begin{pmatrix}0&\dbar E_6(u,v)\ee^{-2\ii t(\theta(u+\ii v;z_0)-\theta(z_0;z_0))}\\ 0&0\end{pmatrix}
,&\quad u+\ii v\in\Omega_6.
\end{cases}
\label{eq:Delta-def}
\end{equation}
Note that $\mathbf{W}(u,v;x,t)$ has jump discontinuities across the sector boundaries in general.
\item[]\textit{\textbf{Normalization:}} $\mathbf{E}(u,v)\to\mathbb{I}$ as $(u,v)\to\infty$.
\end{itemize}
\label{dbp:E}
\end{dbarproblem}
To show the continuity, first note that in each of the six sectors $\Omega_j$, $j=1,\dots,6$, $\mathbf{E}(u,v;x,t)$ is continuous as a function of $(u,v)$ up to the sector boundary.  Indeed, the first factor in \eqref{eq:E-define} is independent of $(u,v)$, and the second factor in \eqref{eq:E-define} has the claimed continuity because this is a property of the solution $\mathbf{N}(u+\ii v;x,t)$ of Riemann-Hilbert Problem~\ref{rhp:N} and of the change-of-variables formula \eqref{eq:O-N}.  Finally, $\mathbf{P}(\zeta;m)$ has unit determinant and its explicit formula in terms of parabolic cylinder functions shows that its restriction to each sector is an entire function of $\zeta$, which guarantees the asserted continuity of the third factor in \eqref{eq:E-define}.  Moreover, the matrices $(2t^{1/2})^{\ii\nu(z_0)\sigma_3}\mathbf{O}(u,v;x,t)$ and $\mathbf{P}(2t^{1/2}(u+\ii v-z_0);|r(z_0)|)$ satisfy exactly the same jump conditions across the six rays that form the common boundaries of neighboring sectors, from which it follows that $\mathbf{E}_+(u,v;x,t)=\mathbf{E}_-(u,v;x,t)$ holds across each of these rays and therefore $\mathbf{E}(u,v;x,t)$ may be regarded as a continuous function of $(u,v)\in\mathbb{R}^2$.  

To show that $\dbar\mathbf{E}=\mathbf{E}\mathbf{W}$ holds, one simply differentiates $\mathbf{E}(u,v;x,t)$ in each of the six sectors, using the fact that $\mathbf{O}(u,v;x,t)$ is related to $\mathbf{N}(u+\ii v;x,t)$ explicitly by \eqref{eq:O-N} and that both $\mathbf{N}(u+\ii v;x,t)$ and the unit-determinant matrix function $\mathbf{P}(2t^{1/2}(u+\ii v-z_0);|r(z_0)|)$ are analytic functions of $u+\ii v$ in each sector, and hence are annihilated by $\dbar$.  The region of non-analyticity of $\mathbf{E}$ is therefore the union of shaded sectors shown in Figure~\ref{fig:RHP-deform}.

Finally to show the normalization condition, we recall Lemma~\ref{lem:decay}.  Therefore, comparing the normalization conditions of Riemann-Hilbert Problem~\ref{rhp:N} for $\mathbf{N}(z;x,t)$ and of Riemann-Hilbert Problem~\ref{rhp:PC} for $\mathbf{P}(\zeta;m)$ shows that $\mathbf{E}(u,v;x,t)\to\mathbb{I}$ as $(u,v)\to\infty$ in $\mathbb{R}^2$.

The rest of this section is devoted to the proof of the following result.
\begin{prop}
Suppose that $r\in H^1(\mathbb{R})$ with $|r(z)|\le\rho$ for some $\rho<1$.  If $t>0$ is sufficiently large, then for all $x\in\mathbb{R}$ there exists a unique solution $\mathbf{E}(\cdot,\cdot;x,t)\in L^\infty(\mathbb{R}^2)$ of $\dbar$-Problem~\ref{dbp:E} with the property that 
\begin{equation}
\mathbf{E}_1(x,t):=\mathop{\lim_{(u,v)\to\infty}}_{u=0}(u+\ii v)\left[\mathbf{E}(u,v;x,t)-\mathbb{I}\right]
\label{eq:E1-define}
\end{equation}
exists and satisfies 
\begin{equation}
\sup_{x\in\mathbb{R}}\|\mathbf{E}_1(x,t)\|=\mathcal{O}(t^{-3/4}),\quad t\to +\infty.
\label{eq:E1-decay}
\end{equation}
\label{prop:dbar}
\end{prop}

\begin{proof}
To show that $\dbar$-Problem~\ref{dbp:E} has a unique solution for $t>0$ sufficiently large, and simultaneously obtain estimates for the solution $\mathbf{E}(u,v;x,t)$, we formulate a weakly-singular integral equation whose solution is that of $\dbar$-Problem~\ref{dbp:E}:
\begin{equation}
\mathbf{E}(u,v;x,t)=\mathbb{I} +\mathcal{J}\mathbf{E}(u,v;x,t),\quad\mathcal{J}\mathbf{F}(u,v):=-\frac{1}{\pi}\iint_{\mathbb{R}^2}\frac{\mathbf{F}(U,V)\mathbf{W}(U,V;x,t)}{(U-u)+\ii(V-v)}\,\dd A(U,V),
\label{eq:integral-equation}
\end{equation}
in which the identity matrix $\mathbb{I}$ is viewed as a constant function on $\mathbb{R}^2$.  Indeed,
this is a consequence of the distributional identity $\dbar z^{-1}=-\pi\delta$ where $\delta$ denotes the Dirac mass at the origin.  
We will solve the integral equation \eqref{eq:integral-equation} in the space $L^\infty(\mathbb{R}^2)$, by computing the corresponding operator norm\footnote{All $L^p$ norms of matrix-valued functions in this section depend on the choice of matrix norm, which we always take to be induced by a norm on $\mathbb{C}^2$.} of $\mathcal{J}:L^\infty(\mathbb{R}^2)\to L^\infty(\mathbb{R}^2)$ and showing that for large $t>0$ it is less than $1$.  Thus, we begin with the elementary estimate
\begin{equation}
\|\mathcal{J}\mathbf{F}(u,v)\|\le \frac{1}{\pi}\|\mathbf{F}\|_{L^\infty(\mathbb{R}^2)}\iint_{\mathbb{R}^2}
\frac{\|\mathbf{W}(U,V;x,t)\|\,\dd A(U,V)}{\sqrt{(U-u)^2+(V-v)^2}}.
\label{eq:JF-absolute-value}
\end{equation}
Using the uniform boundedness of $\mathbf{P}(\zeta;m)$ and its inverse with respect to $\zeta$:  there exists $C>0$ such that $\|\mathbf{P}(\zeta;m)\|\le C$ and $\|\mathbf{P}(\zeta;m)^{-1}\|\le C$ for all $\zeta\in\mathbb{C}\setminus\Sigma_\mathbf{P}$ and all $m\in [0,\rho]$ with $\rho<1$, the assumption $|r(z)|\le\rho<1$ gives that 
\begin{equation}
\|\mathbf{W}(u,v;x,t)\|\le C^2\begin{cases}\ee^{-8t(u-z_0)v}|\dbar E_1(u,v)|,&\quad z=u+\ii v\in\Omega_1\\
\ee^{8t(u-z_0)v}|\dbar E_3(u,v)|,&\quad z=u+\ii v\in\Omega_3\\
\ee^{-8t(u-z_0)v}|\dbar E_4(u,v)|,&\quad z=u+\ii v\in\Omega_4\\
\ee^{8t(u-z_0)v}|\dbar E_6(u,v)|,&\quad z=u+\ii v\in\Omega_6,
\end{cases}
\label{eq:W-estimate}
\end{equation}
and of course $\mathbf{W}(u,v;x,t)\equiv 0$ on $\Omega_2\cup\Omega_5$.  By direct computation using \eqref{eq:dbar-extensions} along with the analyticity of $f(z;z_0)^{\pm 2}$ provided by Lemma~\ref{lem:f} and straightforward estimates of $\cos(2\arg(u+\ii v-z_0))$ and its $\dbar$-derivative as in Section~\ref{sec:linear}, we have the following analogues of \eqref{eq:dbarE-1}:
\begin{equation}
|\dbar E_1(u,v)|
\le \frac{1}{2}|f(u+\ii v;z_0)^{-2}||r'(u)| + \frac{|f(u+\ii v;z_0)^{-2}r(u)-r(z_0)|}{\sqrt{(u-z_0)^2+v^2}},\quad z=u+\ii v\in\Omega_1,
\end{equation}
\begin{multline}
|\dbar E_3(u,v)|\le\frac{1}{2}|f(u+\ii v;z_0)^2|\left|\frac{\dd}{\dd u}\frac{\overline{r(u)}}{1-|r(u)|^2} \right|\\
{}+\left|\frac{f(u+\ii v;z_0)^2\overline{r(u)}}{1-|r(u)|^2}-\frac{\overline{r(z_0)}}{1-|r(z_0)|^2}\right|\frac{1}{\sqrt{(u-z_0)^2+v^2}},\quad z=u+\ii v\in\Omega_3,
\end{multline}
\begin{multline}
|\dbar E_4(u,v)|\le\frac{1}{2}|f(u+\ii v;z_0)^{-2}|\left|\frac{\dd}{\dd u}\frac{r(u)}{1-|r(u)|^2}\right|
\\{}+\left|\frac{f(u+\ii v;z_0)^{-2}r(u)}{1-|r(u)|^2}-\frac{r(z_0)}{1-|r(z_0)|^2}\right|\frac{1}{\sqrt{(u-z_0)^2+v^2}},\quad z=u+\ii v\in\Omega_4,
\end{multline}
and
\begin{equation}
|\dbar E_6(u,v)|\le\frac{1}{2}|f(u+\ii v;z_0)^2||r'(u)|+\frac{|f(u+\ii v;z_0)^2\overline{r(u)}-\overline{r(z_0)}|}{\sqrt{(u-z_0)^2+v^2}},\quad z=u+\ii v\in\Omega_6.
\end{equation}
Note that
\begin{equation}
\left|\frac{\dd}{\dd u}\frac{\overline{r(u)}}{1-|r(u)|^2}\right|=\left|\frac{\dd}{\dd u}\frac{r(u)}{1-|r(u)|^2}\right|\le\frac{1+\rho^2}{(1-\rho^2)^2}|r'(u)|
\label{eq:derivative}
\end{equation}
holds under the condition $|r(u)|\le\rho<1$.  Also, under the same condition,
\begin{equation}
\begin{split}
|f(u+\ii v;z_0)^{-2}r(u)-r(z_0)|&=|(f(u+\ii v;z_0)^{-2}-1)r(u) + r(u)-r(z_0)|\\ &\le \rho|f(u+\ii v;z_0)^{-2}-1| + |r(u)-r(z_0)|\\
&\le \left(K\rho+\|r'\|_{L^2(\mathbb{R})}\right)[(u-z_0)^2+v^2]^{1/4}, 
\end{split}
\end{equation}
where we used Lemma~\ref{lem:f} and \eqref{eq:qhat-difference}, and $K>0$ depends on $\rho$ but not on $z_0$.  Exactly the same estimate holds for $|f(u+\ii v;z_0)^2\overline{r(u)}-\overline{r(z_0)}|$.
In the same way, but also using \eqref{eq:derivative}, 
\begin{equation}
\begin{split}
\left|\frac{f(u+\ii v;z_0)^2\overline{r(u)}}{1-|r(u)|^2}-\frac{\overline{r(z_0)}}{1-|r(z_0)|^2}\right|&\le
\left(\frac{K\rho}{1-\rho^2}+\frac{1+\rho^2}{(1-\rho^2)^2}\|r'\|_{L^2(\mathbb{R})}\right)[(u-z_0)^2+v^2]^{1/4}\\
\left|\frac{f(u+\ii v;z_0)^{-2}r(u)}{1-|r(u)|^2}-\frac{r(z_0)}{1-|r(z_0)|^2}\right|&\le
\left(\frac{K\rho}{1-\rho^2}+\frac{1+\rho^2}{(1-\rho^2)^2}\|r'\|_{L^2(\mathbb{R})}\right)[(u-z_0)^2+v^2]^{1/4}.
\end{split}
\end{equation}
Therefore again using Lemma~\ref{lem:f}, we see that there are constants $L$ and $M$ depending only on the upper bound $\rho<1$ for $\|r\|_{L^\infty(\mathbb{R})}$, on $\|r\|_{L^2(\mathbb{R})}$, and on $\|r'\|_{L^2(\mathbb{R})}$ such that
\begin{equation}
|\partial E_j(u,v)|\le L|r'(u)| +\frac{M}{[(u-z_0)^2+v^2]^{1/4}},\quad z=u+\ii v\in\Omega_j,\quad j=1,3,4,6.
\label{eq:dbarE-nonlinear}
\end{equation}
Note that \eqref{eq:dbarE-nonlinear} is the nonlinear analogue of the estimate \eqref{eq:dbarE-2}.

Combining \eqref{eq:dbarE-nonlinear} with \eqref{eq:JF-absolute-value}--\eqref{eq:W-estimate} shows that
for some constant $D$ independent of $(x,t)\in\mathbb{R}^2$, 
\begin{equation}
\|\mathcal{J}\mathbf{F}(u,v;x,t)\|\le D\left[I^{[1,4]}(u,v;x,t) + J^{[1,4]}(u,v;x,t) + I^{[3,6]}(u,v;x,t) +J^{[3,6]}(u,v;x,t)\right] \|\mathbf{F}\|_{L^\infty(\mathbb{R}^2)},
\label{eq:JF-pointwise}
\end{equation}
where the four terms are analogues in the nonlinear case of the double integrals defined in \eqref{eq:linear-integrals} for the linear case:
\begin{equation}
\begin{split}
I^{[1,4]}(u,v;x,t)&:=\iint_{\Omega_1\cup\Omega_4}\frac{|r'(U)|\ee^{-8t(U-z_0)V}\,\dd A(U,V)}{\sqrt{(U-u)^2+(V-v)^2}},\\
I^{[3,6]}(u,v;x,t)&:=\iint_{\Omega_3\cup\Omega_6}\frac{|r'(U)|\ee^{8t(U-z_0)V}\,\dd A(U,V)}{\sqrt{(U-u)^2+(V-v)^2}},\\
J^{[1,4]}(u,v;x,t)&:=\iint_{\Omega_1\cup\Omega_4}\frac{\ee^{-8t(U-z_0)V}\,\dd A(U,V)}{[(U-z_0)^2+V^{2}]^{1/4}\sqrt{(U-u)^2+(V-v)^2}},\quad\text{and}\\
J^{[3,6]}(u,v;x,t)&:=\iint_{\Omega_3\cup\Omega_6}\frac{\ee^{8t(U-z_0)V}\,\dd A(U,V)}{[(U-z_0)^2+V^{2}]^{1/4}\sqrt{(U-u)^2+(V-v)^2}}.
\end{split}
\label{eq:nonlinear-integrals}
\end{equation}
Estimation of the integrals $I^{[1,4]}(u,v;x,t)$ and $J^{[1,4]}(u,v;x,t)$ requires nearly identical steps as estimation of $I^{[3,6]}(u,v;x,t)$ and $J^{[3,6]}(u,v;x,t)$ (just note that the sign of the exponent always corresponds to decay in the sectors of integration).  So for brevity we just deal with $I^{[3,6]}(u,v;x,t)$ and $J^{[3,6]}(u,v;x,t)$.

To estimate $I^{[3,6]}(u,v;x,t)$, by iterated integration we have
\begin{equation}
\begin{split}
I^{[3,6]}(u,v;x,t)&=\left[\int_0^{+\infty}\,\dd V\int_{-\infty}^{z_0-V}\,\dd U +\int_{-\infty}^0\,\dd V\int_{z_0-V}^{+\infty}\,\dd U\right]\frac{|r'(U)|\ee^{8t(U-z_0)V}}{\sqrt{(U-u)^2+(V-v)^2}}\\
&\le\left[\int_{0}^{+\infty}\,\dd V\int_{-\infty}^{z_0-V}\,\dd U + \int_{-\infty}^0\,\dd V\int_{z_0-V}^{+\infty}\,\dd U\right]\frac{|r'(U)|\ee^{-8tV^2}}{\sqrt{(U-u)^2+(V-v)^2}}.
\end{split}
\label{eq:I36-iterated-integrals}
\end{equation}
The inner integrals can be estimated by Cauchy-Schwarz, using the fact that $r'\in L^2(\mathbb{R})$:
\begin{equation}
\begin{split}
\pm\int_{\mp\infty}^{z_0-V}\frac{|r'(U)|\,\dd U}{\sqrt{(U-u)^2+(V-v)^2}}&\le\int_{\mathbb{R}}\frac{|r'(U)|\,\dd U}{\sqrt{(U-u)^2+(V-v)^2}}\\
&\le\|r'\|_{L^2(\mathbb{R})}\sqrt{\int_\mathbb{R}\frac{\dd U}{(U-u)^2+(V-v)^2}} = \frac{\|r'\|_{L^2(\mathbb{R})}\sqrt{\pi}}{\sqrt{|V-v|}}.
\end{split}
\end{equation}
Thus, 
\begin{equation}
I^{[3,6]}(u,v;x,t)\le\|r'\|_{L^2(\mathbb{R})}\sqrt{\pi}\int_\mathbb{R}\frac{\ee^{-8tV^2}\,\dd V}{\sqrt{|V-v|}}.
\label{eq:I36-1}
\end{equation}
Without loss of generality, suppose that $v> 0$.  Then
\begin{equation}
\int_\mathbb{R}\frac{\ee^{-8tV^2}\,\dd V}{\sqrt{|V-v|}} = \int_{-\infty}^0\frac{\ee^{-8tV^2}\,\dd V}{\sqrt{v-V}} + \int_0^v\frac{\ee^{-8tV^2}\,\dd V}{\sqrt{v-V}} +\int_v^{+\infty}\frac{\ee^{-8tV^2}\,\dd V}{\sqrt{V-v}}.
\label{eq:V-integral-sum}
\end{equation}
Using monotonicity of $\sqrt{v-V}$ on $V<0$ and the rescaling $V=t^{-1/2}w$, we get for the first term:
\begin{equation}
\int_{-\infty}^0\frac{\ee^{-8tV^2}\,\dd V}{\sqrt{v-V}}\le\int_{-\infty}^0\frac{\ee^{-8tV^2}\,\dd V}{\sqrt{-V}} = t^{-1/4}\int_{-\infty}^0\frac{\ee^{-8w^2}\,\dd w}{\sqrt{-w}} = \mathcal{O}(t^{-1/4}).
\label{eq:V-integral-1}
\end{equation}
For the second term, we use the inequality $\ee^{-b}\le Cb^{-1/4}$ for $b>0$ and the rescaling $V=vw$ to get
\begin{equation}
\int_0^v\frac{\ee^{-8tV^2}\,\dd V}{\sqrt{v-V}}\le C(8t)^{-1/4}\int_0^v\frac{\dd V}{\sqrt{V(v-V)}}
= C(8t)^{-1/4}\int_0^1\frac{\dd w}{\sqrt{w(1-w)}}=\mathcal{O}(t^{-1/4}).
\label{eq:V-integral-3}
\end{equation}
Using monotonicity of $\ee^{-8tV^2}$ on $V>v$ and the change of variable $V-v=t^{-1/2}w$ we get for the third term:
\begin{equation}
\int_v^{+\infty}\frac{\ee^{-8tV^2}\,\dd V}{\sqrt{V-v}}\le \int_v^{+\infty}\frac{\ee^{-8t(V-v)^2}\,\dd V}{\sqrt{V-v}} = t^{-1/4}\int_0^{+\infty}\frac{\ee^{-8w^2}\,\dd w}{\sqrt{w}}=\mathcal{O}(t^{-1/4}).
\label{eq:V-integral-2}
\end{equation}
The upper bounds in \eqref{eq:V-integral-1}--\eqref{eq:V-integral-3} are all independent of $v$ (and $u$), so combining them with \eqref{eq:I36-1}--\eqref{eq:V-integral-sum} gives
\begin{equation}
\sup_{(u,v)\in\mathbb{R}^2} I^{[3,6]}(u,v;x,t)\le C\|r'\|_{L^2(\mathbb{R})} t^{-1/4},
\label{eq:I36-2}
\end{equation}
where $C$ denotes an absolute constant.

To estimate $J^{[3,6]}(u,v;x,t)$ we again introduce iterated integrals in the same way as in \eqref{eq:I36-iterated-integrals} to obtain the inequality
\begin{equation}
J^{[3,6]}(u,v;x,t)\le\left[\int_0^{+\infty}\,\dd V\int_{-\infty}^{z_0-V}\,\dd U +
\int_{-\infty}^0\,\dd V\int_{z_0-V}^{+\infty}\,\dd U\right]\frac{\ee^{-8tV^2}}{[(U-z_0)^2+V^2]^{1/4}
\sqrt{(U-u)^2+(V-v)^2}}.
\end{equation}
Now, to estimate the inner $U$-integrals we will use H\"older's inequality with conjugate exponents $p>2$ and $q<2$.  Thus,
\begin{multline}
\pm\int_{\mp\infty}^{z_0-V}\frac{\dd U}{[(U-z_0)^2+V^2]^{1/4}\sqrt{(U-u)^2+(V-v)^2}}\\
\begin{aligned}
&\le\left(\pm\int_{\mp\infty}^{z_0-V}\frac{\dd U}{[(U-z_0)^2+V^2]^{p/4}}\right)^{1/p}
\left(\pm\int_{\mp\infty}^{z_0-V}\frac{\dd U}{[(U-u)^2+(V-v)^2]^{q/2}}\right)^{1/q}\\
&\le\left(\int_\mathbb{R}\frac{\dd U}{[(U-z_0)^2+V^2]^{p/4}}\right)^{1/p}
\left(\int_\mathbb{R}\frac{\dd U}{[(U-u)^2+(V-v)^2]^{q/2}}\right)^{1/q}.
\end{aligned}
\end{multline}
Now, by the change of variable $U-z_0=|V|w$,
\begin{equation}
\left(\int_\mathbb{R}\frac{\dd U}{[(U-z_0)^2+V^2]^{p/4}}\right)^{1/p}=
|V|^{1/p-1/2}\left(\int_\mathbb{R}\frac{\dd w}{[w^2+1]^{p/4}}\right)^{1/p},
\end{equation}
where the integral on the right-hand side is convergent as long as $p>2$.  Similarly, by the change of variable $U-u=|V-v|w$,
\begin{equation}
\left(\int_\mathbb{R}\frac{\dd U}{[(U-u)^2+(V-v)^2]^{q/2}}\right)^{1/q} = 
|V-v|^{1/q-1}\left(\int_\mathbb{R}\frac{\dd w}{[w^2+1]^{q/2}}\right)^{1/q},
\end{equation}
where the integral on the right-hand side is convergent as long as $q>1$.  Hence for any conjugate exponents $1<q<2<p<\infty$ with $p^{-1}+q^{-1}=1$, we have for some constant $C=C(p,q)$,
\begin{equation}
J^{[3,6]}(u,v;x,t)\le C\int_\mathbb{R}\ee^{-8tV^2}|V|^{1/p-1/2}|V-v|^{1/q-1}\,\dd V.
\label{eq:J36-1}
\end{equation}
As before, assume without loss of generality that $v> 0$.  Then
\begin{multline}
\int_\mathbb{R}\ee^{-8tV^2}|V|^{1/p-1/2}|V-v|^{1/q-1}\,\dd V=
\int_{-\infty}^0\ee^{-8tV^2}(-V)^{1/p-1/2}(v-V)^{1/q-1}\,\dd V \\{}+
\int_0^v\ee^{-8tV^2}V^{1/p-1/2}(v-V)^{1/q-1}\,\dd V +\int_v^{+\infty}\ee^{-8tV^2}V^{1/p-1/2}(V-v)^{1/q-1}\,\dd V.
\label{eq:J-integral-sum}
\end{multline}
Using $q>1$ and monotonicity of $(v-V)^{1/q-1}$ on $V<0$ along with $1/p+1/q=1$ and the rescaling $V=t^{-1/2}w$ gives
for the first integral
\begin{equation}
\begin{split}
\int_{-\infty}^0\ee^{-8tV^2}(-V)^{1/p-1/2}(v-V)^{1/q-1}\,\dd V&\le
\int_{-\infty}^0\ee^{-8tV^2}(-V)^{1/p-1/2 +1/q-1}\,\dd V\\ &=
\int_{-\infty}^0\ee^{-8tV^2}(-V)^{-1/2}\,\dd V \\ &= t^{-1/4}\int_{-\infty}^0\ee^{-8w^2}(-w)^{-1/2}\,\dd w = \mathcal{O}(t^{-1/4}).
\end{split}
\label{eq:J-integral-1}
\end{equation}
For the second integral, we again recall $\ee^{-b}\le Cb^{-1/4}$ for $b>0$ and rescale by $V=vw$ to get
\begin{equation}
\begin{split}
\int_0^v\ee^{-8tV^2}V^{1/p-1/2}(v-V)^{1/q-1}\,\dd V&\le C(8t)^{-1/4}\int_0^v V^{1/p-1}(v-V)^{1/q-1}\,\dd V\\
&= C(8t)^{-1/4}\int_0^1 w^{1/p-1}(1-w)^{1/q-1}\,\dd w =\mathcal{O}(t^{-1/4}),
\end{split}
\label{eq:J-integral-2}
\end{equation}
using also $q,p<\infty$.  Finally, for the third integral, we use monotonicity of $\ee^{-8tV^2}$ and $V^{1/p-1/2}$ (for $p>2$) on $V>v$ and make the substitution $V-v=t^{-1/2}w$ to get
\begin{equation}
\begin{split}
\int_v^{+\infty}\ee^{-8tV^2}V^{1/p-1/2}(V-v)^{1/q-1}\,\dd V &\le
\int_v^{+\infty}\ee^{-8t(V-v)^2}(V-v)^{1/p-1/2}(V-v)^{1/q-1}\,\dd V\\ &=
\int_v^{+\infty}\ee^{-8t(V-v)^2}(V-v)^{-1/2}\,\dd V\\ &=
t^{-1/4}\int_0^{+\infty}\ee^{-8w^2}w^{-1/2}\,\dd w = \mathcal{O}(t^{-1/4}).
\end{split}
\label{eq:J-integral-3}
\end{equation}
Since the upper bounds in \eqref{eq:J-integral-1}--\eqref{eq:J-integral-3} are all independent of $(u,v)\in\mathbb{R}^2$, combining them with \eqref{eq:J36-1}--\eqref{eq:J-integral-sum} gives
\begin{equation}
\sup_{(u,v)\in\mathbb{R}^2}J^{[3,6]}(u,v;x,t)\le C t^{-1/4},
\label{eq:J36-2}
\end{equation}
where $C$ denotes an absolute constant.

Returning to \eqref{eq:JF-pointwise} and taking a supremum over $(u,v)\in\mathbb{R}^2$, we see that 
\begin{equation}
\|\mathcal{J}\mathbf{F}\|_{L^\infty(\mathbb{R}^2)}\le Dt^{-1/4}\|\mathbf{F}\|_{L^\infty(\mathbb{R}^2)},\quad\text{\emph{i.e.,}}\quad \|\mathcal{J}\|_{L^\infty(\mathbb{R}^2)\circlearrowleft}\le Dt^{-1/4}
\label{eq:J-estimate}
\end{equation}
holds where $D$ is a constant depending only on the upper bound $\rho<1$ for $\|r\|_{L^\infty(\mathbb{R})}$, on $\|r\|_{L^2(\mathbb{R})}$, and on $\|r'\|_{L^2(\mathbb{R})}$, and where $\|\mathcal{J}\|_{L^\infty(\mathbb{R}^2)\circlearrowleft}$ denotes the norm of the weakly-singular integral operator $\mathcal{J}$ acting in $L^\infty(\mathbb{R}^2)$.  It is a consequence of \eqref{eq:J-estimate}  that the integral equation \eqref{eq:integral-equation} is uniquely solvable in $L^\infty(\mathbb{R}^2)$ by convergent Neumann series for sufficiently large $t>0$:
\begin{equation}
\mathbf{E}(u,v;x,t)=(\mathcal{I}-\mathcal{J})^{-1}\mathbb{I}=\mathbb{I}+\mathcal{J}\mathbb{I} + \mathcal{J}^2\mathbb{I}+\mathcal{J}^3\mathbb{I}+\cdots,\quad t>D^{-4},
\label{eq:Neumann}
\end{equation}
where $\mathcal{I}$ denotes the identity operator and $\mathbb{I}$ the constant function on $\mathbb{R}^2$, and that the solution satisfies 
\begin{equation}
\|\mathbf{E}-\mathbb{I}\|_{L^\infty(\mathbb{R}^2)}\le \frac{Dt^{-1/4}}{1-Dt^{-1/4}}=\mathcal{O}(t^{-1/4}),\quad t\to +\infty,
\label{eq:EminusI}
\end{equation}
an estimate that is uniform with respect to $x\in\mathbb{R}$.  This proves the first assertion in Proposition~\ref{prop:dbar}.

To prove the existence of the limit $\mathbf{E}_1(x,t)$ in \eqref{eq:E1-define}, note that
from the integral equation \eqref{eq:integral-equation} we have
\begin{equation}
\begin{split}
(u+\ii v)\left[\mathbf{E}(u,v;x,t)-\mathbb{I}\right]&=\frac{1}{\pi}\iint_{\mathbb{R}^2}\mathbf{E}(U,V;x,t)\mathbf{W}(U,V;x,t)\,\dd A(U,V)\\
&\quad\quad -\frac{1}{\pi}\iint_{\mathbb{R}^2}\frac{U+\ii V}{(U-u)+\ii (V-v)}\mathbf{E}(U,V;x,t)\mathbf{W}(U,V;x,t)\,\dd A(U,V).
\end{split}
\label{eq:two-terms}
\end{equation}
The second term satisfies 
\begin{multline}
\left\|\iint_{\mathbb{R}^2}\frac{U+\ii V}{(U-u)+\ii (V-v)}\mathbf{E}(U,V;x,t)\mathbf{W}(U,V;x,t)\,\dd A(U,V)\right\|\\
\le \|\mathbf{E}\|_{L^\infty(\mathbb{R}^2)}\iint_{\mathbb{R}^2}\sqrt{\frac{U^2+V^2}{(U-u)^2+(V-v)^2}}\|\mathbf{W}(U,V;x,t)\|\,\dd A(U,V).
\label{eq:Dom-Conv-1}
\end{multline}
Now, following \cite{LiuPS18}, let us examine the resulting double integral for $u=0$, \emph{i.e.}, for $z=u+\ii v$ restricted to the imaginary axis.  Some simple trigonometry shows that 
\begin{equation}
\sup_{(U,V)\in\mathrm{supp}(\mathbf{W}(\cdot,\cdot;x,t))}\sqrt{\frac{U^2+V^2}{U^2 + (V-v)^2}} = 1+\sqrt{2}\frac{|v|}{|v|-|z_0|},\quad |v|>|z_0|.
\label{eq:kernel-Linfty}
\end{equation}
Therefore, if $u=0$, the double integral on the right-hand side of \eqref{eq:Dom-Conv-1} will tend to zero as $|v|\to\infty$ by the Lebesgue dominated convergence theorem provided that $\mathbf{W}(\cdot,\cdot;x,t)\in L^1(\mathbb{R}^2)$.  Using \eqref{eq:W-estimate} and \eqref{eq:dbarE-nonlinear}, we have
\begin{equation}
\iint_{\mathbb{R}^2}\|\mathbf{W}(U,V;x,t)\|\,\dd A(U,V)\le D\left[\widetilde{I}^{[1,4]}(x,t) +
\widetilde{J}^{[1,4]}(x,t) + \widetilde{I}^{[3,6]}(x,t) +\widetilde{J}^{[3,6]}(x,t)\right],
\end{equation}
where (compare with \eqref{eq:nonlinear-integrals}, or better yet, \eqref{eq:linear-integrals})
\begin{equation}
\begin{split}
\widetilde{I}^{[1,4]}(x,t)&:=\iint_{\Omega_1\cup\Omega_4}|r'(U)|\ee^{-8t(U-z_0)V}\,\dd A(U,V),\\
\widetilde{I}^{[3,6]}(x,t)&:=\iint_{\Omega_3\cup\Omega_6}|r'(U)|\ee^{8t(U-z_0)V}\,\dd A(U,V),\\
\widetilde{J}^{[1,4]}(x,t)&:=\iint_{\Omega_1\cup\Omega_4}\frac{\ee^{-8t(U-z_0)V}\,\dd A(U,V)}{[(U-z_0)^2+V^2]^{1/4}},\quad\text{and}\\
\widetilde{J}^{[3,6]}(x,t)&:=\iint_{\Omega_3\cup\Omega_6}\frac{\ee^{8t(U-z_0)V}\,\dd A(U,V)}{[(U-z_0)^2+V^2]^{1/4}}.
\end{split}
\end{equation}
Noting the resemblance with the double integrals \eqref{eq:linear-integrals} analyzed in Section~\ref{sec:linear}, we can immediately obtain the estimate
\begin{equation}
\iint_{\mathbb{R}^2}\|\mathbf{W}(U,V;x,t)\|\,\dd A(U,V)\le Ct^{-3/4}<\infty
\label{eq:W-L1}
\end{equation}
for some constant $C$ independent of $x$. Therefore, the second term on the right-hand side of \eqref{eq:two-terms} tends to zero as $v\to\infty$ if $u=0$ (the limit is not uniform with respect to $x$ since $v$ is compared with $z_0$ in \eqref{eq:kernel-Linfty}).  Comparing with \eqref{eq:E1-define}, we obtain from \eqref{eq:two-terms} the formula
\begin{equation}
\mathbf{E}_1(x,t):=\frac{1}{\pi}\iint_{\mathbb{R}^2}\mathbf{E}(U,V;x,t)\mathbf{W}(U,V;x,t)\,\dd A(U,V),
\label{eq:E1-formula}
\end{equation}
and exactly the same argument shows that $\mathbf{E}_1(x,t)$ is finite and uniformly decaying as $t\to +\infty$:
\begin{equation}
\begin{split}
\|\mathbf{E}_1(x,t)\|\le\frac{1}{\pi}\|\mathbf{E}\|_{L^\infty(\mathbb{R}^2)}\|\mathbf{W}\|_{L^1(\mathbb{R}^2)}&\le \frac{1}{\pi}\left(\|\mathbb{I}\|_{L^\infty(\mathbb{R}^2)}+\|\mathbf{E}-\mathbb{I}\|_{L^\infty(\mathbb{R}^2)}\right)\|\mathbf{W}\|_{L^1(\mathbb{R}^2)}\\
&\le\frac{C}{\pi}\left(1+\frac{Dt^{-1/4}}{1-Dt^{-1/4}}\right)t^{-3/4} = \mathcal{O}(t^{-3/4}),
\end{split}
\end{equation}
where we have used \eqref{eq:EminusI} and \eqref{eq:W-L1} and noted that the constants $C$ and $D$ are independent of $x$.  This proves the second assertion in Proposition~\ref{prop:dbar}.
\end{proof}

\subsection{The solution of the Cauchy problem \eqref{eq:NLSEQ}--\eqref{eq:IC} for $t>0$ large}
\label{sec:q}
Now we complete the proof of Theorem~\ref{mainresult} by combining our previous results.
The matrix function $\mathbf{N}(u+\ii v;x,t)$ agrees with $\mathbf{O}(u,v;x,t)$ for $u=0$ and $|v|$ sufficiently large given $z_0=-x/(4t)$.  Since according to \eqref{eq:E-define}, 
\begin{equation}
\mathbf{O}(u,v;x,t)=(2t^{1/2})^{-\ii\nu(z_0)\sigma_3}\mathbf{E}(u,v;x,t)\mathbf{P}(2t^{1/2}(u+\ii v-z_0);|r(z_0)|),
\end{equation}
we compute the matrix coefficient $\mathbf{N}_1(x,t)$ appearing in \eqref{eq:qNrelation} by taking a limit along the imaginary axis in \eqref{eq:N-normalize}.  Thus, we obtain $\mathbf{N}_1(x,t)=(2t^{1/2})^{-\ii\nu(z_0)\sigma_3}\mathbf{Q}(x,t)(2t^{1/2})^{\ii\nu(z_0)\sigma_3}$, where
\begin{equation}
\begin{split}
\mathbf{Q}(x,t)&=(2t^{1/2})^{\ii\nu(z_0)\sigma_3}\left\{\mathop{\lim_{(u,v)\to\infty}}_{u=0} (u+\ii v)\left[\mathbf{N}(u+\ii v;x,t)(z-z_0)^{-\ii\nu(z_0)\sigma_3}-\mathbb{I}\right]\right\}(2t)^{-\ii\nu(z_0)\sigma_3}\\
&=\mathop{\lim_{(u,v)\to\infty}}_{u=0}(u+\ii v)\left[\mathbf{E}(u,v;x,t)\mathbf{P}(2t^{1/2}(u+\ii v-z_0);|r(z_0)|)(2t^{1/2}(u+\ii v-z_0))^{-\ii\nu(z_0)\sigma_3}-\mathbb{I}\right].
\end{split}
\end{equation}
Using \eqref{eq:P-expansion} and Proposition~\ref{prop:dbar} yields
\begin{equation}
\mathbf{Q}(x,t)=\mathbf{E}_1(x,t) + \frac{1}{2}t^{-1/2}\mathbf{P}_1(|r(z_0)|).
\end{equation}
Therefore, using \eqref{eq:qNrelation} gives the following formula for the solution of the Cauchy problem \eqref{eq:NLSEQ}--\eqref{eq:IC}:
\begin{equation}
\begin{split}
q(x,t)&=2\ii\ee^{-\ii\omega(z_0)}\ee^{-2\ii t\theta(z_0;z_0)}c(z_0)^{-2}(2t^{1/2})^{-2\ii\nu(z_0)}Q_{12}(x,t)\\
&= \ee^{-\ii\omega(z_0)}\ee^{-2\ii t\theta(z_0;z_0)}c(z_0)^{-2}(2t^{1/2})^{-2\ii\nu(z_0)}\left[2\ii E_{1,12}(x,t) + \frac{1}{2}t^{-1/2}2\ii P_{1,12}(|r(z_0)|)\right]\\
&= \ee^{-\ii\omega(z_0)}\ee^{-2\ii t\theta(z_0;z_0)}c(z_0)^{-2}(2t^{1/2})^{-2\ii\nu(z_0)}\left[2\ii E_{1,12}(x,t) + \frac{1}{2}t^{-1/2}\beta(|r(z_0)|)\right],
\end{split}
\label{eq:nl-q-formula}
\end{equation}
where we recall $\omega(z_0)=\arg(r(z_0))$, $\theta(z_0;z_0)=-2z_0^2$, the definition \eqref{eq:nu-define} of $\nu(z_0)$, the definition \eqref{eq:c-define} of $c(z_0)$, and the definitions \eqref{eq:beta-mod-squared} and \eqref{eq:arg-beta} of $|\beta(m=|r(z_0)|)|^2$ and $\arg(\beta(m=|r(z_0)|))$ respectively.  Since the factors to the left of the square brackets have unit modulus, from 
Proposition~\ref{prop:dbar} it follows that $q(x,t)$ has exactly the representation \eqref{eq:AsForm}
in which $|\mathcal{E}(x,t)|=|E_{1,12}(x,t)|=\mathcal{O}(t^{-3/4})$ as $t\to +\infty$, uniformly with respect to $x$.  This completes the proof of Theorem~\ref{mainresult}.

\subsection*{Remark}
The use of truncations of the Neumann series \eqref{eq:Neumann} for $\mathbf{E}(u,v;x,t)$ yields a corresponding asymptotic expansion of $q(x,t)$ as $t\to+\infty$.  In other words, it is straightforward (but tedious) to compute explicit corrections to the leading term in the asymptotic formula \eqref{eq:AsForm} by expanding $\mathcal{E}(x,t)$.  For instance, the formula \eqref{eq:E1-formula} gives
\begin{equation}
\mathbf{E}_1(x,t)=\frac{1}{\pi}\iint_{\mathbb{R}^2}\mathbf{W}(U,V;x,t)\,\dd A(U,V) + 
\frac{1}{\pi}\iint_{\mathbb{R}^2}(\mathbf{E}(U,V;x,t)-\mathbb{I})\mathbf{W}(U,V;x,t)\,\dd A(U,V),
\end{equation}
\emph{i.e.}, an explicit double integral plus a remainder.  Using the estimates \eqref{eq:EminusI} and \eqref{eq:W-L1} we find that the remainder term satsifies
\begin{equation}
\begin{split}
\sup_{x\in\mathbb{R}}\left\|\frac{1}{\pi}\iint_{\mathbb{R}^2}(\mathbf{E}(U,V;x,t)-\mathbb{I})\mathbf{W}(U,V;x,t)\,\dd A\right\| &\le
\frac{1}{\pi}\sup_{x\in\mathbb{R}}\|\mathbf{E}(\cdot,\cdot;x,t)\|_{L^\infty(\mathbb{R}^2)}\|\mathbf{W}(\cdot,\cdot;x,t)\|_{L^1(\mathbb{R})}\\ & = \mathcal{O}(t^{-1/4}t^{-3/4})=\mathcal{O}(t^{-1}),\quad t\to +\infty.
\end{split}
\end{equation}
Using this result in \eqref{eq:nl-q-formula} gives in place of \eqref{eq:AsForm} the corrected asymptotic formula
\begin{equation}
q(x,t)=q^{(0)}(x,t) + q^{(1)}(x,t)  + \mathcal{E}^{(1)}(x,t)
\label{eq:q-expansion}
\end{equation}
where
\begin{equation}
q^{(0)}(x,t):=t^{-1/2}\alpha(z_0)\ee^{\ii x^2/(4t)-\ii\nu(z_0)\ln(8t)} 
\end{equation}
is the leading term in \eqref{eq:AsForm},
\begin{equation}
q^{(1)}(x,t):= \frac{2\ii}{\pi}\ee^{-\ii\omega(z_0)}\ee^{-2\ii t\theta(z_0;z_0)}c(z_0)^{-2}(2t^{1/2})^{-2\ii\nu(z_0)}\iint_{\mathbb{R}^2}W_{12}(U,V;x,t)\,\dd A(U,V)
\label{eq:q1}
\end{equation}
is an explicit correction (see \eqref{eq:W-def}--\eqref{eq:Delta-def}), and where $\mathcal{E}^{(1)}(x,t)$ is error term satisfying $\mathcal{E}^{(1)}(x,t)=\mathcal{O}(t^{-1})$ as $t\to +\infty$ uniformly with respect to $x\in\mathbb{R}$.  Theorem~\ref{mainresult} implies that the correction satisfies $\|q^{(1)}(\cdot,t)\|_{L^\infty(\mathbb{R})}=\mathcal{O}(t^{-3/4})$ as $t\to+\infty$, but the explicit formula \eqref{eq:q1} 
allows for a complete analysis of the correction.  For instance, we are in a position to seek reflection coefficients $r(z)$ in the Sobolev space $H^1(\mathbb{R})$ with $|r(z)|\le\rho<1$ for which the correction saturates the upper bound of $\mathcal{O}(t^{-3/4})$, or to determine under which conditions on $r(z)$ the correction term can be smaller.  Under additional hypotheses the expansion \eqref{eq:q-expansion} can be carried out to higher order, with subsequent corrections involving iterated double integrals of $\mathbf{W}$, which in turn involve $\dbar$-derivatives of the extensions $E_j$, $j=1,3,4,6$, and the parabolic cylinder functions contained in the matrix $\mathbf{P}(\zeta;m)$ solving Riemann-Hilbert Problem~\ref{rhp:PC}.

\end{document}